\documentclass[12pt, twoside]{article}
\usepackage{amsmath,amsthm,amssymb}
\usepackage{times}
\usepackage{enumerate,comment}
\usepackage{color}
\usepackage{hyperref}

\def\titlerunning#1{\gdef\titrun{#1}}
\makeatletter
\def\author#1{\gdef\autrun{\def\and{\unskip, }#1}\gdef\@author{#1}}
\def\address#1{{\def\and{\\\hspace*{18pt}}\renewcommand{\thefootnote}{}%
\footnote {#1}}%
\markboth{\autrun}{\titrun}}
\makeatother
\def\email#1{e-mail: #1}
\def\subjclass#1{{\renewcommand{\thefootnote}{}%
\footnote{\emph{Mathematics Subject Classification (2010):} #1}}}

\newcommand{\Pbb}[1]{\Pb\lb #1\rb}
\newcommand{\Qbb}[1]{\Qb\lb #1\rb}
\newcommand{\Ebb}[1]{\Eb\lbb #1\rbb}

\newcommand{\LL}{L\'{e}vy }

\newcommand{\PP}{\overline{\Pi}}

\newcommand{\tto}[1]{_{#1\to 0}}

\newcommand{\ttinf}[1]{_{#1\to \infty}}

\newcommand{\Eb}{\mathbb{E}}
\newcommand{\Nb}{\mathbb{N}}
\newcommand{\Rb}{\mathbb{R}}
\newcommand{\Pb}{\mathbb{P}}
\newcommand{\Qb}{\mathbb{Q}}

\newcommand{\Bc}{\mathcal{B}}
\newcommand{\Cc}{\mathcal{C}}
\newcommand{\Dc}{\mathcal{D}}
\newcommand{\Ec}{\mathcal{E}}
\newcommand{\Fc}{\mathcal{F}}

\newcommand{\Oc}{\mathcal{O}}

\newcommand{\Tc}{\mathcal{T}}

\newcommand{\Wc}{\mathcal{W}}

\newcommand{\ind}[1]{1_{#1}}

\newcommand{\lb}{\left (}
\newcommand{\rb}{\right )}
\newcommand{\lbb}{\left [}
\newcommand{\rbb}{\right ]}
\newcommand{\labs}{\left |}
\newcommand{\rabs}{\right |}
\newcommand{\lbrb}[1]{\lb #1 \rb}
\newcommand{\lbbrbb}[1]{\lbb#1\rbb}
\newcommand{\labsrabs}[1]{\labs#1\rabs}

\newcommand{\lbcurly}{\left\{}
\newcommand{\rbcurly}{\right\}}
\newcommand{\lbcurlyrbcurly}[1]{\lbcurly#1\rbcurly}
\newcommand{\twopartdef}[4]
{
	\left\{
	\begin{array}{ll}
		#1 & \mbox{if } #2 \\
		#3 & \mbox{if } #4
	\end{array}
	\right.
}

\newcommand{\gN}[1]{\sqrt{g(#1)}}

\newtheorem{proposition}{Proposition}
\newtheorem {lemma}{Lemma}
\newtheorem {theorem}{Theorem}
\newtheorem {Corollary}{Corollary}

\theoremstyle{definition}

\newtheorem{remark}[theorem]{Remark}
\numberwithin{equation}{section}
\numberwithin{equation}{section}
\numberwithin{equation}{section}

\numberwithin{equation}{section}

\begin{document}


\baselineskip=17pt


\titlerunning{Integral Test and Repulsion Envelope}

\title{Transience and recurrence of a Brownian path with limited local time and its repulsion envelope}

\author{Martin Kolb
\and 
Mladen Savov}

\date{}

\maketitle

\address{M. Kolb: Department of Mathematics and Statistics,
  University of Reading, Whiteknights, Reading RG6 6AX, UK; \email{M.Kolb@reading.ac.uk}
\and
M. Savov: Department of Mathematics and Statistics,
  University of Reading, Whiteknights, Reading RG6 6AX, UK; \email{m.savov@reading.ac.uk}}

\begin{abstract}
In this note we investigate the behaviour of Brownian motion conditioned on a growth constraint of its local time which has been previously investigated by Berestycki and Benjamini. For a class of non-decreasing positive functions $f(t), t\geq 0$, we consider the Wiener measure under the condition that the Brownian local time is dominated by the function $f$ up to time $T$. In the case where $f(t)/t^{3/2}$ is integrable we describe the limiting process as $T \rightarrow \infty$. Moreover, we prove two conjectures in \cite{BB1} in the case for a class of functions $f$, for which $f(t)/t^{3/2}$ just fails to be integrable. Our methodology is more general as it relies on the study of the uniform asymptotic of the probability of subordinators to stay above a given family of curves. Immediately, one can study questions like the Brownian motioned conditioned on a growth constraint of its local time at the maximum.
\end{abstract}

\subjclass{Primary 60J55 ; Secondary 60G17, 60J65}


\section{Introduction}
Let $(B_t)_{t \geq 0}$ be a one dimensional standard Brownian motion. In this paper, by developing a very general methodology for studying the asymptotic of the probability of increasing \LL processes (subordinators) to stay above a given curve, we study the behaviour of Brownian paths, which have a limited growth of local time at the origin. Following previous work \cite{BB1} of Berestycki and Benjamini we consider the problem of describing the measures 
\begin{displaymath}
\mathbb{P}_t:=\mathbb{P}\lbrb{\cdot \mid L_s \leq f(s),\, \forall s \leq t}=\Pbb{.\mid \tau_{f(s)}>s,\,\forall s\leq t}
\end{displaymath}
in the limit $t \rightarrow \infty$, where $(L_s)_{s \geq 0}$ denotes the local time of $(B_t)_{t\geq 0}$ at the origin, $(\tau_s)_{s\geq 0}$ its right-inverse and $f : [0,\infty)\rightarrow [0,\infty)$ is a suitable non-negative increasing function satisfying some additional mild properties.

Let us now describe the main results of \cite{BB1} in more detail. It is shown that the family of probability measures $\mathbb{P}_t$ on the canonical path space $\mathcal{C}=C([0, \infty),\mathbb{R})$ is in fact tight and thus has limit points. Furthermore, the authors manage to show that the condition 
\begin{equation}\label{eq:mainCondition}
I(f)=\int_1^{\infty}\frac{f(t)}{t^{\frac{3}{2}}}\,dt < \infty
\end{equation}
implies that every weak limit point $\mathbb{Q}$ of $\mathbb{P}_t$, as $t \rightarrow \infty$, is transient almost surely. This means in particular that restricting the local time growth to be smaller than $f(t)=\sqrt{t}(\log t)^{-1-\epsilon},\epsilon>0,$ already results in a significant change of the original recurrent Brownian motion. Observe that this might be surprising as the typical growth of the local time coincides with $\sqrt{t}$ and thus we only require a slightly slower than average growth.

This intriguing result immediately leads to the question whether the tight family $\mathbb{P}_t$ is in fact weakly convergent and whether one can in some way interpret its limit. Exactly this question leads to one part of the present contribution.

In \cite{BB1} it is conjectured that \eqref{eq:mainCondition} is the precise dividing line between every possible weak limit of $\Pb_t$ being recurrent or transient. In our current work we show that this integral distinguishes between recurrence and transience by elaborating a method which captures all classes of functions $f$ such that $\lim\ttinf{t}f(t)\ln^{4/5+\epsilon}(t)/t^{1/2}=0$, for some $\epsilon>0$, when $I(f)=\infty$ and $\lim\ttinf{t}f(t)\ln^{1/2}(t)/t^{1/2}=0$ when $I(f)<\infty$. Given that the functions $f(t)=t^{1/2}/\ln(t)$ and $f(t)=t^{1/2}/\ln^{1+\epsilon}(t)$ are on the two sides of the integral test \eqref{eq:mainCondition} we see that our restriction is irrelevant for the most critical region. We further develop the results of \cite{BB1} in several different directions:
\begin{itemize}
\item First we show that $\Qb=\lim\ttinf{t}\Pb_t$ exists. We explicitly identify the limit in the case $I(f)<\infty$ and further prove that it corresponds to a recurrent process  if $I(f)=\infty$. This settles two questions left open in \cite{BB1}. 
\item Furthermore, motivated by Conjecture 2 of \cite{BB1} we say that an increasing function $w$ is in the repulsion envelope of $f$ if even $\lim\ttinf{t}\Qb\lbrb{L_t<f(t)/w(t)}=1$. Using our methods we manage to analytically describe the repulsion envelope of $f$ by providing a simple explicit criterion which provides a necessary and sufficient condition (NASC) for $w$ to be in the envelope of $f$. The quantifies the idea of entropic repulsion which is often used in the physics literature. 
\end{itemize}

Observe that the general scheme of conditioning on an unlikely event has similarities to papers on quasistationary distribtuions (see e.g. \cite{C}), penalizations (see e.g.\cite{RY}) as well as to approaches investigated in the area of polymer models (see e.g. \cite{vdHK}, \cite{MS08} and \cite{N}). The questions considered in this paper (as well as our methods) still differ from the just mentioned ones in the sense that one of our main aims is to study the phenomenon of entropic repulsion in a simple but still highly non-trivial situation. This phenomenon has  already been the main topic of several previous studies such as \cite{BB}, \cite{BB1} and \cite{BGMS} and usually refers to the fact that conditioning on an unlikely event often results in a process whose behaviour appears to be even more unlikely than the one which the process is conditioned on. In our setting the phenomenon of entropic repulsion is most clearly visible in Theorem 4.4 which proves that the repulsion envelope is not empty. 

Let us describe the structure of the paper. In the next section we set up the problem, the notation, present some basic facts that will be used later, provide a short discussion on the strategy of the proof {and discuss the scope of our methodology}. In Section \ref{sec:transience} we consider the case $I(f)<\infty$ and describe the limiting process and prove that it is transient. In Section \ref{sec:recurrence} under mild assumptions we discuss the case when $I(f)=\infty$ and show that the limiting process exists and is recurrent which is solves Conjecture 1 in \cite{BB1}. Additionally, we determine analytically the repulsion envelope showing that it is never empty thus settling Conjecture 2 in \cite{BB1}. In Section \ref{sec:differentialEqnandEstimates} we provide the basic ODE which allows us to estimate in various ways the quantity $\Pbb{\Oc_t}$ namely the probability of the event we condition on. The last parts are devoted to the proofs.

\section{Notation and Discussion}\label{sec:notation}
\subsection{Basic notation}\label{subsec:basicNotation}
We use throughout the paper the following conventions. First we use $f\sim g$ to denote that $\lim\ttinf{t}\frac{f(t)}{g(t)}=1$ and $\mu_t(ds)\sim \nu_t(ds)$ to denote that the densities $m_t(s),v_t(s)$ of the measures $\mu_t,\nu_t$ satisfy $\lim\ttinf{t}\frac{m_t(s)}{v_t(s)}=1$ for each finite $s>0$ where we preclude the possibility of $v_t(s)=0$. Similarly, we use $f\asymp g$ to denote the existence of two constants $0<D_1<D_2<\infty$ such that $$D_1\leq \liminf\ttinf{t}\frac{f(t)}{g(t)}\leq\limsup\ttinf{t}\frac{f(t)}{g(t)}\leq D_2$$
with the same meaning for measures indexed by $t$ at the level of their densities, see above for $\sim$. The notation $f\lesssim g$ respectively $f\gtrsim f$ then imply the existence of $D_2$ respectively $D_1$ above. 

Throughout the paper we also use the convention that $C$ will be an absolute positive and finite constant, whereas $C(A,B,\cdots)$ will denote an absolute constant not depending of any variables but $A,B,\cdots$. 
\subsection{The boundary function $f(x)$ and its inverse $g(x)$}\label{subsec:f}
Without loss of generality we will assume that $f(1)=1$, $1>f(0)>0$ and that $f:\Rb\mapsto \Rb^+$ is an increasing function which drifts to infinity. We impose the following mild growth condition:
\begin{equation}\label{eq:growthCondition}
(0,\infty) \ni x\mapsto \frac{f(x)}{\sqrt{x}}\, \text{ is decreasing and } \lim_{x\rightarrow \infty}\frac{f(x)}{\sqrt{x}} = 0.
\end{equation}
 Often we work with $g(x):=f^{-1}(x)$ for which \eqref{eq:mainCondition} is with the help of \eqref{eq:growthCondition} translated to 
\begin{equation}\label{eq:mainConditionG}
I(f)<\infty \iff J(g):=\int_{1}^{\infty}\frac{1}{\gN{s}}ds<\infty.
\end{equation}
Observe that we can continuously and monotonously extend the function $g$ to the interval $(0,f(0))$. Note that since $f(0)>0$ we have that for $x\in (0,f(0))$, $g(x)<0$.

\subsection{Brownian motion, Local time, inverse Local time and related quantities}
In this paper we work with a standard Brownian motion $B=\lbrb{B_s}_{s\geq 0}$. Recall that for a real-valued Brownian motion the local time at zero can be defined as
\[L_t=\lim\tto{\varepsilon}\frac{1}{2\varepsilon}\int_{0}^{t}1_{\{B_s\in\lbrb{-\varepsilon,\varepsilon}\}}ds.\]
The local time is a continuous, non-decreasing process which grows precisely on the set $\{s\geq 0: B_s=0\}$.
It is well known that its right-inverse local time $\tau=\lbrb{\tau_t}_{t\geq 0}$, where 
\begin{equation}\label{eq:tau}
\tau_u=\inf\{t>0:\,L_t>u\}
\end{equation}
 is a stable subordinator of index $1/2$, i.e. a non-decreasing \LL process without drift whose \LL measure is given by $\Pi(ds)=Kds/s^{3/2}, s>0$, where $K:=1/\sqrt{2\pi}$, and the \LL-Khintchine exponent of $\tau_1$ is given by
\begin{equation}\label{eq:LevyKhintchine}
\Ebb{e^{-\lambda\tau_1}}=e^{-K\int_{0}^{\infty}\lbrb{1-e^{-\lambda s}}\frac{ds}{s^{3/2}}}=e^{-K\sqrt{\lambda}\int_{0}^{\infty}\lbrb{1-e^{-s}}\frac{ds}{s^{3/2}}},
\end{equation}
where in view of working with subordinators which are obtained from $\tau$ by truncating some of its jumps we do not compute explicitly $\int_{0}^{\infty}\lbrb{1-e^{-\lambda s}}\frac{ds}{s^{3/2}}$ as any truncation will be reflected in the region of integration.

Furthermore,  the law and the density of $\tau_u$ can be computed via 
\begin{align}\label{eq:lawTau}
\nonumber&\Pbb{\tau_u>t}=\Pbb{\tau_1>t/u^2}=\frac{u}{\sqrt{2\pi t}}\int_{-1}^{1}e^{-\frac{(ux)^2}{2t}}dx=\frac{2u}{\sqrt{\pi}}\int_{0}^{\frac{1}{\sqrt{2t}}}e^{-(ur)^2}dr,\quad\text{ for $t>0$}\\
&\Pbb{\tau_u\in dt}=\frac{ue^{-\frac{u^2}{2t}}}{\sqrt{2\pi t^3}}dt=f_u(t)dt,\quad\text{ for $t>0$}.
\end{align}
Then with $f$, $g$ given above we see that 
\begin{equation}\label{eq:equivalenceToExpectation}
I(f)<\infty\iff J(g)<\infty\iff \Ebb{f(\tau_1)}<\infty.
\end{equation} 

Note that the jumps of the subordinator $\tau$ correspond to the lengths of the excursions of the Brownian motion away from zero, which is due to \eqref{eq:tau} and therefore the fact that $L_t$ is a constant on span of each excursion away from zero. In more technical detail, the excursions are paths in $\Cc$ with the following properties: $\varepsilon\in\Cc$, $\varepsilon(0)=0$, $\varepsilon(t)>0$ or $\varepsilon(t)<0$, for $\forall t<\zeta(\varepsilon)$, $\varepsilon(t)=0$, $t\geq \zeta(\varepsilon)$, where $\zeta$ is called the length or life-time of the excursion and determines a jump in the subordinator $\tau$. We refer to \cite{BB1} for a very good exposition of excursions for this setting and \cite[Chapter IV]{B97} for more general \LL processes.

Finally, we denote by $\lbrb{\Fc_t}_{t\geq 0}$ the natural filtration of the inverse local time $\tau$ which is via standard random time change generated by the natural filtration of the Brownian motion.

\subsection{The event on which the process is conditioned}
Throughout the paper it will be convenient to work with the inverse local time $\tau$. We use that the following sets are equal:

\begin{equation}\label{eq:O}
\Oc_t=\{\tau_s>g(s);\,s\leq t\}=\{L_{g(s)}\leq s;\, s\leq t\}=\{L_{u}\leq f(u);\, u\leq g(t)\},\,\text{ $\forall t>0$.}
\end{equation}
This definition slightly differs from the sets $\Ec_t=\Oc_{f(t)}$ used in \cite{BB1}. This difference is irrelevant for the limit.
 
Important functions in our study will be $\phi(t)=\Phi'(t)=\Pbb{\Oc_t}$, where
\begin{equation}\label{eq:Phi}
\Phi(t)=\int_{0}^{t}\Pbb{\Oc_s}ds.
\end{equation}
In Section \ref{sec:differentialEqnandEstimates} we provide the explicit asymptotic behaviour of $\phi(t)$ and $\Phi(t)$ via an ordinary linear differential equation of first order which links $\phi$ and $\Phi$. These are the results at the heart of our main theorems. One might find it surprising that such precise estimates can be given for such highly dependent events. In fact $\Oc_t$ depends on the whole path of the process $\tau$ up to time $t$.

\subsection{Discussion and strategy for the proof}\label{subsec:strategyProof}
Since we condition on $\Oc_t$ the results naturally depend on the knowledge about the asymptotic of $\phi(t)=\Pbb{\Oc_t}$. Since $\phi(t)$ and $\Phi(t)$ are linked by a linear differential equation \eqref{eq:differentialEquation} of the type
\begin{equation*}\label{eq:differentialEquation}
\phi(t)-\frac{2K}{\sqrt{\varphi(t)}}\Phi(t)=H(t),
\end{equation*}
This can be solved and the function $H(t)$ can be estimated rather precisely.  This allows us to provide very sharp results on the asymptotic of $\phi(t)$ and $\Phi(t)$. The differential equation itself arises by simply conditioning on the time of a first jump of $\tau$ that will take $\tau$ above $g(t)$ which removes the dependence on the future. The fact that we can estimate $H$ comes from the one-large jump principle which roughly states that one large jump determines the large deviation behaviour of $\tau$. Since $\lim\ttinf{t}g(t)/t^2=\infty$ and by scaling $\Pbb{\tau_t>ct^2}=\Pbb{\tau_1>c}$ we see that we are in the regime of large deviations for $\Pbb{\tau_t>g(t)}$ and the one-large jump principle is expected to hold true for $\Pbb{\Oc_t}$. However, this is a harder to verify in our scenario as $\Oc_t$ depends on the entire past of the process.  

Due to the heavy space-time dependence revealed for example by 
\[\Pbb{\tau_h\in dy;\,\Oc_t}=\Pbb{\tau_{s}>g(s+h)-y;\,\forall 0\leq s\leq t-h}\Pbb{\tau_h\in dy;\Oc_h},\]
where the function $g(s)\mapsto g(s+h)-y$, information on $\phi(t)$ and $\Phi(t)$ does not suffice. Using the same differential equation \eqref{eq:differentialEquation} for each point $(h,y)$ and function $g_{y,h}(s):=g(s+h)-y$ we are able to prove some uniform bounds for $\phi^h_y(t)$ and $\Phi^h_y(t)$. When $I(f)<\infty$ these bounds do not require heavy calculations. In the situation $I(f)=\infty$ these are much harder. Precisely in this case we need the condition $\lim\ttinf{t}f(t)\ln^{4/5+\epsilon}(t)/t^{1/2}=0$ but we have no doubt that the exponent $4/5$ can be made much smaller. However, this would unnecessary burden the exposition of the paper and our condition in any case captures the transition from $I(f)<\infty\to I(f)=\infty$.

Once suitable bounds on  $\phi^h_y(t)$ and $\Phi^h_y(t)$ are settled, it is a matter of dominated convergence theorem and the tightness of $\Pbb{\tau_h\in\,.\mid \Oc_t}$ to show that in the two scenarios $I(f)<\infty$ and $I(f)=\infty$ the limiting process exists and it is correspondingly transient and recurrent. However, the estimates can be used even further. An estimate of the quantity $\Qbb{\tau_h\in\lbrb{g(h),g(h)w(h)}}$ can be made very precise and analytical which allows us to prove a NASC for $\lim\ttinf{h}\Qbb{\tau_h\in\lbrb{g(h),g(h)w(h)}}=0$ which in other words distinguishes the functions in the repulsion envelope.

\subsection{Brownian motion conditioned on the growth of its local time at its maximum}
The inverse local time $\tau$ for the reflected Brownian motion $\sup_{s\leq t}B_s-B_t$ has the same law as the inverse local time at zero. Since all our results rely first on the distribution of the inverse local time under the limit measure $\Qb=\lim\ttinf{t}\Pb_t$  and then on splicing of excursions of the Brownian motion we see that all results are of the same type. The difference between the transient and recurrent regime consists in that in the former the all time maximum is obtained in a finite time and then the negative of a Bessel three process is issued forth as in Theorem \ref{thm:transience}. The Bessel three process in this setting occurs in the form of an excursion with infinite life time.

\section{Transient Case}\label{sec:transience}
Recall that $\Cc=C([0, \infty),\mathbb{R})$ is the space of continuous functions indexed by the time $t$ and denote by $\Wc$ the Wiener measure on $\Cc$.

Next define the family of random variables $\mathfrak{C}_t$ with $\text{supp}\, \mathfrak{C}_t=[0,t]$ via their densities as follows
\begin{equation}\label{eq:clocks}
\Pbb{\mathfrak{C}_t\in ds}=\frac{\Pbb{\Oc_s}}{\int_{0}^{t}\Pbb{\Oc_v}dv}ds=\frac{\phi(s)}{\Phi(t)}ds, \text{ for $0<s\leq t$}.
\end{equation} 
Recall that $\Oc_t=\{\tau_s>g(s),\,0\leq s\leq t\}$. Denote by $\Delta^{g(t)}_1=\{s>0\,:\,\tau_s-\tau_{s-}>g(t)\}$. The clocks approximate very precisely the underlying structure namely the fact that the conditions represented by $\Oc_t$ are satisfied with dominating probability by the arrival of one jump larger than $g(t)$, i.e.
\[\Pbb{\Oc_t}\sim\Pbb{\Oc_t\cap \Delta^{g(t)}_1\leq t}.\]
 Conditioned upon arrival on $[0,t]$ the jump has uniform distribution which subsequently is reweighted in \eqref{eq:clocks} to reflect the additional assumption that $\Oc_s$ must hold until time $\Delta^{g(t)}_1$. This size of the large jump for $\tau$ is in fact the length of an excursion of the Brownian motion away from zero. In the limit this excursion conditioned to last more than $g(t)$ converges to the three dimensional Bessel process, see for example \cite[p.10--11]{BB1}, though this is a standard result in the probability folklore. 

When $\mathfrak{C}=\lim\ttinf{t}\mathfrak{C}_t$ exists in a weak sense, namely iff $\Phi(\infty)<\infty$, then it has a density function $$\Pbb{\mathfrak{C}\in ds}=\frac{\Pbb{\Oc_s}}{\Phi(\infty)}ds,\,s\geq 0.$$ Define the process $(Y_t)_{t \geq 0}$ in the following way: choose independent copies of the clock $\mathfrak{C}$; of $B=\lbrb{B_s}_{s\geq 0}$; of $\varpi\in \{-1,1\}$ with $\Pbb{\varpi=1}=1/2$; and of $B^{(3)}=\lbrb{B^{(3)}_s}_{s\geq0 }$, where $B^{(3)}$ is a three dimensional Bessel process; then 
\begin{enumerate}
\item Conditionally on $\{\mathfrak{C}=x\}$ run $B$ conditioned on $\{L_s\leq f(s);\,s\leq \tau_x\}$ (note that $L_{\tau_x}=x$) and put $Y_s$ to coincide with this conditioned process for $s\leq \tau_x$. 
\item Choose $1$ or $-1$ according to $\varpi$;
\item For $t>\tau_x$ put $\lbrb{Y_{t}}_{t\geq \tau_x}=\lbrb{\varpi B^{(3)}_{t-\tau_x}}_{t\geq \tau_x}$.
\end{enumerate}
The next result shows that under $\Qb$, $B$ equals precisely $Y$ whenever $I(f)<\infty$.
\begin{theorem}\label{thm:transience}
Assume that $f$ is as defined in part \ref{subsec:f}, \eqref{eq:mildCondition} and $I(f)<\infty$. Then  $\mathfrak{C}=\lim\ttinf{t}\mathfrak{C}_t$ exists in a weak sense and furthermore $\Qbb{.}=\lim\ttinf{t}\Pb_t(.)$ in the sense of the weak topology on the space $\Cc$. Under the measure $\Qb$, the process $B$ equals the process $Y$. Moreover, for any fixed $h>0$ and any $y>g(h)\vee 0$ we have the formula for the density of $\tau$ under $\Qb$
\begin{equation}\label{eq:tauTransience}
\Qbb{\tau_h\in dy}:=\frac{\Phi^h_y(\infty)}{\Phi(\infty)}\Pbb{\tau_h\in dy;\Oc_h},
\end{equation}
where $\Phi^h_y(\infty)=\int_{0}^{\infty}\Pbb{\tau_s>g(s+h)-y,\,s\leq v}dv<\infty$ is part of the claim. Therefore, $\Qbb{\tau_h<\infty}=\Pbb{\mathfrak{C}> h}$. Finally, for any measurable $\Bc\subset\Oc_h$ and $\Bc\in\Fc_h$ we have that
\begin{equation}\label{eq:h-transformTransience}
\Qbb{\Bc}=\Ebb{\frac{\Phi^h_{\tau_h}(\infty)}{\Phi(\infty)};\Bc}.
\end{equation} 
\end{theorem}
\begin{remark}
Note that this result is consistent with \cite[Theorem 2]{BB} where $f(s)\equiv1$ is studied despite that the inverse function $g$ is undefined. The clock there is a uniform random variable on $(0,1)$ and the local time is accumulated until this random variable is attained. Then a Bessel process with random sign is issued forth. The Bessel process is a result of the limit of longer and longer excursions away from zero which in turn are a consequence of the one-large jump principle.
\end{remark}
\begin{Corollary}\label{cor:transience}
We have that under $Q$ the process $B$ is transient, namely $\Qbb{\lim\ttinf{t}|B_t|=\infty}=1$ and even  $|B_s|_{s\geq\tau_{\mathfrak{C}-}}=B^{(3)}_{s-\tau_{\mathfrak{C}-}}$. Therefore after time $\tau_\mathfrak{C}$ the process is explicit and its density and rate of growth, which determines the speed of transience are computed as those of the three dimensional Bessel process.
\end{Corollary}
We proceed with the recurrence case.
\section{Recurrent Case}\label{sec:recurrence}
\subsection{Weak limit and recurrence}
The recurrent case is much more demanding. We will impose the following condition as it suffices to capture the transition region:
\begin{equation}\label{eq:conditionGrowth}
\liminf\ttinf{t}\frac{g(t)}{t^2 \ln^{\frac{8}{5}+\epsilon}(t)}=\infty, \text{ for some $\epsilon>0.$}
\end{equation}
Condition \eqref{eq:conditionGrowth} undoubtedly can be further relaxed but this would require more precision in the heavy computations below and will add less value as we have already captured the transitions region with \eqref{eq:conditionGrowth}.

Under weaker condition $\liminf\ttinf{t}g(t)/t^{2}\ln(t)=\infty$ we can see from \eqref{eq:finitenessPhi}, Lemma \ref{lem:finitenessPhi} that the limit clock $\mathfrak{C}$ defined in Section \ref{sec:transience} does not exist since $\Phi(\infty)=\infty$. This in turn is a good indicator as to why the recurrence holds: still
\[\Pbb{\Oc_t}\sim\Pbb{\Oc_t\cap \Delta^{g(t)}_1\leq t},\] 
but, for any $a<\infty$,
\[\lim\ttinf{t}\frac{\Pbb{\Oc_t\cap \Delta^{g(t)}_1\leq a}}{\Pbb{\Oc_t}}=0,\]
 see \eqref{eq:clocks}, when $\Phi(\infty)=\infty$, and the long excursion, which is the cause of the Bessel process to appear in the transient scenario, is pushed away to infinity with probability one.

We have the following statement.
\begin{theorem}\label{thm:recurrent}
Let $f$ satisfy the usual conditions in part \ref{subsec:f}. Additionally, assume that $I(f)=\infty$ and \eqref{eq:growthCondition} holds. Then the limit  $\lim\ttinf{t}\Pbb{.\Big|\Oc_t}=\lim\ttinf{t}\Pb_t(.)=\Qb(.)$ exists and under $\Qb$ the process is recurrent, namely $$\Qbb{\exists t>T:\,B_t=0}=1,\,\,\forall T>0$$
Under $\Pb_t$ the inverse local time converges, as $t\to\infty$, to an increasing pure-jump process under $\Qb$ which we call the inverse local time under $\Qb$.
\end{theorem}
The increasing pure-jump process referred to in the above theorem is studied in more detail in Proposition \ref{cor:densityUnderQ}.

We proceed to utilize this information and discuss the phenomena of repulsion. 
\subsection{Repulsion envelope}
Let us define the set of functions $\Dc=\{w:[1,\infty)\mapsto [1,\infty):\text{$w$ is increasing to $\infty$}\}$ and $R_g=\Dc\cap\{w: \lim\ttinf{h}\Qbb{\tau_h\geq w(h)g(h)}=1\} $. We call $R_g$  the envelope of repulsion which means that in fact under $\Qb$ the inverse local time stays with increasing to one probability not only above $g$ but above $g_w:=gw$. Note that if $f_w=g^{-1}_w$ then solving for $u=f/f_w$ we see that $u\downarrow 0$ is such that $\lim\ttinf{t}\Qbb{L_t\leq f(t)u(t)}=1$. It is conjectured in \cite[Conjecture 2]{BB1} that $R_g\neq \emptyset$ with some further quite insightful comments as to the form of functions that comprise $R_g$. Our next result shows that one in fact can in a simple analytical way specify $R_g$. We are able to do this thanks to \eqref{eq:densityInverseLocalTime}. We have the following statement.

\begin{theorem}\label{thm:repulsion}
Let the conditions upon $f$ of Theorem \ref{thm:recurrent} hold. Let $w\in \Dc$ then we have 
\begin{equation}\label{eq:repulsionRegion}
w\in R_g \iff \lim\ttinf{h} \int_{h}^{f\lbrb{g(h)w(h)}}\frac{1}{\gN{s}}ds=0.
\end{equation}
\end{theorem}
\begin{remark}
Take a function $f(t)=\sqrt{t}/\ln^{\gamma}(t)$ with $1>\gamma>4/5$ then we have that $g(t)\sim t^2\ln^{2\gamma}(t)$. Define $w_\gamma(t)=e^{\ln^\gamma(t)}$ and $g_{w_\gamma}=gw_\gamma$ and then easily $g^{-1}_{w_\gamma}(t)=: f_{w_\gamma}(t)\sim e^{-\kappa \ln^\gamma(t)}\sqrt{t}/\ln^\gamma(t)$, as $t\to\infty$ and some $\kappa >0$. Then using \eqref{eq:repulsionRegion} we can see the conjectured function $w_\gamma(t)$ is indeed the separating line of $R_g$ since for any $w$ such that $\ln{w}=o(\ln{w_\gamma})$ then $w\in R_g$ but in fact $w_\gamma\notin R_g$. Computing \eqref{eq:repulsionRegion} explicitly we can even have the simplified criterion $w\in R_g$ iff $\ln(w(h))=o(\ln^\gamma(h))$.
\end{remark}
\begin{remark}
The case $\gamma=1$ is the most interesting as it correspond to the case $g(t)\sim t^2\ln^2(t)$ at the boundary of our transition region. Then an easy computation yields that
\begin{align*}
&\lim\ttinf{h}\int_{h}^{f\lbrb{g(h)w(h)}}\frac{1}{\gN{s}}ds=\\
&\lim\ttinf{h}\ln\lbrb{\frac{\ln(h)+\ln(w(h))+\ln\ln(h)}{\ln(h)}}=0 \iff \ln(w(h))=o(\ln(h)).
\end{align*}
\end{remark}
\begin{remark}
We would like to point out that due to the fact that we estimate many quantities with constants bounded away from zero it will be difficult to study other probabilities like $\Qbb{\tau_h\in (g(h),g(h)w(h)}\to 1$ unless we have a zero-one law something we do not anticipate to be true.
\end{remark}
\section{Precise asymptotic estimates for $\Pbb{\Oc_t}$ and $\int\limits_{0}^{t}\Pbb{\Oc_s}ds$.}\label{sec:differentialEqnandEstimates}
The fact that $\tau$ is a stable subordinator and thus enjoys the so-called one large jump principle allows for the very precise study of the events $\Oc_t=\{\tau_s>g(s),s\leq t\}$ at least to a first order asymptotic. We recall that the one-large jump principle postulates that the probability of the subordinator to cross larger and larger barrier in an also expanding time horizon is asymptotically equivalent to the probability that the subordinator makes one jump of size exceeding the level of this barrier. It is clear that if this principle applies in this setting then the long-term dependency in the definition of $\Oc_t$ will be destroyed at the moment we make a jump bigger than $g(t)$. This is the main observation behind the ensuing estimates.  However, \eqref{eq:differentialEquation} holds in any situation, for any subordinator, and offers the opportunity for more general studies.

Recall that $\Pbb{\Oc_t}=\phi(t)$ and $\Phi(t)=\int_{0}^{t}\Pbb{\Oc_s}ds$. Then the following general  result holds.
\begin{theorem}\label{thm:differentialEquation}
For any function $\varphi(t)\geq g(t)\vee 1$, for $t>0$, we have that
\begin{equation}\label{eq:differentialEquation}
\phi(t)-\frac{2K}{\sqrt{\varphi(t)}}\Phi(t)=H(t),
\end{equation}
where with $\Delta_1^{\varphi(t)}=\inf\{s\geq 0:\tau_s-\tau_{s-}>\varphi(t)\}$ we have that $H(t)$ is defined as follows 
\begin{align}\label{eq:H(t)}
\nonumber &H(t):=\Pbb{\tau_s>g(s),s\leq t;\,\Delta_1^{\varphi(t)}>t}-\\
&\frac{4K^2}{\varphi(t)}\int_{0}^{t}\int_{0}^{s}\Pbb{\tau_u>g(u),u\leq v\,\big|\Delta^{\varphi(t)}_1=v}e^{-\frac{2Kv}{\sqrt{\varphi(t)}}}dvds.
\end{align}
Denote by 
    \begin{equation}\label{eq:rho}
    \rho(t):=\frac{H(t)}{\Phi(t)}.
    \end{equation}
Then, for any $t\geq t_0 \geq 1$, and $\varphi(t)=g(t)$
 \begin{equation}\label{eq:solutionODE}
    \Phi(t)=\Phi(t_0)e^{\int_{t_0}^{t}\frac{2K}{\sqrt{g(s)}}ds+\int_{t_0}^{t}\rho(s)ds}.
    \end{equation}
\end{theorem}
  \begin{remark}
 We have no doubt that the probability of events $\Oc_t$ arising from more general subordinators whose \LL measure tail $\PP(x)=\int_{x}^{\infty}\Pi(ds)$, see \cite[Chapter III]{B97} for more information on subordinatores,  behaves as $ x^{-\alpha}L(x)$, as $x\to\infty$, for some $0<\alpha< 2$ and a slowly varying function $L(x)$, will be amenable to such a study and therefore the main results could be extended to a class of general \LL processes. The conditions for a \LL process to possess a local time at zero and the form of the \LL-Khintchine exponent of the inverse local time can be found in \cite[Chapter V]{B97}.
  \end{remark}
  \begin{remark}
  It is even more interesting to understand whether these equations are applicable only for nondecreasing processes like $\tau$ or a suitable modification can be developed for, say \LL processes. Then the problem of general \LL process $\Pbb{X_s>g(s),\,s\leq t}$ could be attacked with such a simple approach as ODE.
  \end{remark}
  \begin{remark}
  It is important to note that despite that \eqref{eq:differentialEquation} is valid with any $\varphi(t)\geq g(t)\vee 1$ it is most beneficial to work with $g(t)$ itself since then the error term represented by $H(t)$ will be minimal.
  \end{remark}
  \begin{remark}
  We note the striking semblance of the derivation of \eqref{eq:differentialEquation} to the classical renewal theory. Perusing the proof it is apparent that the second term can be decomposed ad infinitum in terms of more and more repeated integrals involving $\Phi(s)$ and further error terms thus obtaining a differential equation involving infinitely many derivatives. 
  \end{remark}

    Assume the following mild technical condition
     \begin{equation}\label{eq:mildCondition}
      \liminf\ttinf{t}\frac{g(t)}{t^{2}\ln(t)}=\infty.
     \end{equation}
      
   From now on we work with $\varphi(t)=g(t)\vee 1$. The next result shows that the finiteness of $\Phi(\infty)$ depends on $I(f)$. We recall the usual conditions \eqref{eq:growthCondition} on $f$:
   \begin{equation*}
   (0,\infty) \ni x\mapsto \frac{f(x)}{\sqrt{x}}\, \text{ is decreasing and } \lim_{x\rightarrow \infty}\frac{f(x)}{\sqrt{x}} = 0.
   \end{equation*}
    \begin{lemma}\label{lem:finitenessPhi}
    Let $f$ satisfy the usual conditions in part \ref{subsec:f} and \eqref{eq:mildCondition}. Then $H(t)=o\lbrb{\Phi(t)/\sqrt{g(t)}}$ and hence $\rho(t)=o\lbrb{1/\sqrt{g(t)}}$.
    Therefore 
    \begin{equation}\label{eq:finitenessPhi}
    \Phi(\infty)<\infty\iff \Ebb{f(\tau_1)}<\infty\iff\int_{1}^{\infty}\frac{ds}{\sqrt{g(s)}}<\infty.
    \end{equation}
    \end{lemma}

 Then equation \eqref{eq:differentialEquation} leads to the following essential result. 
  \begin{theorem}\label{thm:asymptoticFiniteCase}
   For any $f$ satisfying the usual conditions, $I(f)<\infty$ and \eqref{eq:mildCondition} we have that
   \begin{equation}\label{eq:asymptoticFiniteCase}
   \Pbb{\Oc_t}\sim\Pbb{\Oc_t;\,\Delta^{g(t)}_1\leq t} \sim\frac{2K\Phi(\infty)}{\sqrt{g(t)}}, \text{as $t\to\infty$}.
   \end{equation}
 
  \end{theorem} 
  \begin{remark}
  Condition \eqref{eq:mildCondition} is expected to hold when $I(f)<\infty$ unless the function is exceptionally bad.
  \end{remark}
  
  The next result considers the case when $\Phi(\infty)=\infty$. 
     We then have that.
  \begin{theorem}\label{thm:asymptoticInFiniteCase}
    For any $f$ satisfying the conditions in part \ref{subsec:f}, $I(f)=\infty$  and \eqref{eq:mildCondition} we have that, as $t\to\infty$,
    \begin{equation}\label{eq:asymptoticInFiniteCase-1}
    \Pbb{\Oc_t}\sim \Pbb{\Oc_t\cap \Delta^{g(t)}_1\leq t}\sim\frac{2K\Phi(t)}{\sqrt{g(t)}}\,;\,\, \ln\lbrb{\Phi(t)}\sim \int_{1}^{t}\frac{2K}{\sqrt{g(s)}}ds,
    \end{equation}
    where we recall that $\Delta^{g(t)}_1=\inf\lbcurlyrbcurly{t>0:\,\tau_t-\tau_{t-}>g(t)}$.
    Furthermore, if for some $t\geq t_0\geq 1$, $\int_{t_0}^{\infty}\labsrabs{\rho(s)}ds<\infty$,  then \eqref{eq:asymptoticInFiniteCase-1} is augmented to
   \begin{equation}\label{eq:asymptoticInFiniteCase-2}
           \Pbb{\Oc_t}\sim\Pbb{\Oc_t\cap \Delta^{g(t)}_1\leq t}\sim \frac{2K\Phi(t)}{\sqrt{g(t)}}\,;\,\, \Phi(t)\sim \Phi(t_0)e^{\int_{t_0}^{\infty}\rho(s)ds} e^{\int_{t_0}^{t}\frac{2K}{\sqrt{g(s)}}ds}.
           \end{equation}
    In particular, this holds when $\liminf\ttinf{t}g(t)/t^{2}\ln^{8/5+\epsilon}(t)=\infty$, i.e. \eqref{eq:conditionGrowth} holds.
  \end{theorem}
  \begin{remark}
   Note the strong form of the asymptotic \eqref{eq:asymptoticInFiniteCase-2} is essential in the proof of recurrence. As mentioned in Section \ref{subsec:strategyProof} we need to study in a uniform way a family of equations for a generalized form of $\Phi$.
  \end{remark}

  We start by proving the results of this section as they are instrumental in our further analysis.
  
\section{Proof of the results in Section \ref{sec:differentialEqnandEstimates}}
In this section and later we will use the following notation. First we shall attach a superscript to $\Oc_t$, $\tau$, etc. to denote that jumps until given time above certain level are conditioned not have occurred. For example $\Oc^{g(t)}_s=\{\tau^{g(t)}_v>g(v),v\leq s\}$ means that $\Oc_s$ holds for the subordinator $\tau^{g(t)}$ which is constructed from $\tau$ by conditioning that jumps larger than $g(t)$ do not occur. The L\'{e}vy-Khintchine exponent of $\tau_1^{g(t)}$ can be represented by
\begin{align}\label{eq:LevyKhintchineTruncated}
\nonumber \Psi_{g(t)}(\lambda)&=\ln\lbrb{\Ebb{e^{\lambda \tau^{g(t)}_1}}}=K\int_{0}^{g(t)}\lbrb{e^{\lambda s}-1}\frac{ds}{s^{3/2}}=\\
&=K\sqrt{\lambda}\int_{0}^{\lambda g(t)}\lbrb{e^{\lambda s}-1}\frac{ds}{s^{3/2}}, \text{ $\forall\lambda>0,$}
\end{align}
where we note that only the \LL measure $\Pi(ds)=Kds/s^{3/2}$ has been truncated, see \eqref{eq:LevyKhintchine} and that $\tau_1^{g(t)}$ in fact has all exponential moments thus $\Psi_{g(t)}(.)$ is analytic on the complex plane. The analyticity of  $\Psi_{g(t)}(.)$ can be directly read off from the first integral formula in \eqref{eq:LevyKhintchineTruncated} by a power series expansion of the exponential. 

We also use the notation  $\Delta^{a}_k=\inf\{s>\Delta^a_{k-1}:\,\tau_s-\tau_{s-}>a\}$, $\Delta^a_0=0$ to denote the time of the $k^{th}$ jump of $\tau$ larger than $a$. Note that $\Delta^a_1\sim Exp(2K/\sqrt{a})$ where we recall that $\PP(x)=\int_{x}^{\infty}\Pi(ds)=2K/\sqrt{x}$, for all $x>0$, is the intensity measure of the jumps larger than $x$, see \cite{B97} for more information on \LL processes.

We are now ready to start off with our proof.
\begin{proof}[Proof of Theorem \ref{thm:differentialEquation}]

Note that since $\varphi(t)\geq g(t)\vee 1$ we have upon disintegration the values of $\Delta^{\varphi(t)}_1\sim Exp(2K/\sqrt{\varphi(t)})$
\begin{align}\label{eq:equivWithJumps}
&\Pbb{\Oc_t}=\int_{0}^{t}\Pbb{\Oc_t,\Delta^{\varphi(t)}_1\in ds}+\Pbb{\,\Delta^{\varphi(t)}_1>t;\,\Oc_t}=\\
\nonumber &\frac{2K}{\sqrt{\varphi(t)}}\int_{0}^{t}\Pbb{\Oc^{\varphi(t)}_s}e^{-\frac{2Ks}{\sqrt{\varphi(t)}}}ds+\Pbb{\tau_s>g(s),s\leq t;\,\Delta^{\varphi(t)}_1>t}.
\end{align}
Indeed we have that
\[\Pbb{\Oc_t;\,\Delta^{\varphi(t)}_1\in ds}=\Pbb{\Oc_{\Delta^{\varphi(t)}_1};\Delta^{\varphi(t)}_1\in ds}\]
which upon conditioning on $\{\Delta^{\varphi(t)}_1=s\}$ confirms our equation. Next, note that since $\Pbb{\Delta^{\varphi(t)}_1>  s}=e^{-\frac{2Ks}{\sqrt{\varphi(t)}}}$ we obtain that
\[\Pbb{\Oc_s}=\Pbb{\Delta^{\varphi(t)}_1\leq s;\,\Oc_s}+\Pbb{\Delta^{\varphi(t)}_1> s;\,\Oc_s}=\Pbb{\Delta^{\varphi(t)}_1\leq s;\,\Oc_s}+\Pbb{\Oc^{\varphi(t)}_s}e^{-\frac{2Ks}{\sqrt{\varphi(t)}}}.\]
Substituting back for $\Pbb{\Oc^{\varphi(t)}_s}e^{-\frac{2Ks}{\sqrt{\varphi(t)}}}$ we get that
\begin{align*}
&\phi(t)=\Pbb{\Oc_t}=\\
&\frac{2K}{\sqrt{\varphi(t)}}\int_{0}^{t}\Pbb{\Oc_s}ds+\Pbb{\tau_s>g(s),s\leq t;\,\Delta^{\varphi(t)}_1>t}-\frac{2K}{\sqrt{\varphi(t)}}\int_{0}^{t}\Pbb{\Delta^{\varphi(t)}_1\leq s;\,\Oc_s}ds=\\
&\frac{2K}{\sqrt{\varphi(t)}}\Phi(t)+\Pbb{\tau_s>g(s),s\leq t;\,\Delta^{\varphi(t)}_1>t}-\frac{4K^2}{\varphi(t)}\int_{0}^{t}\int_{0}^{s}\Pbb{\Oc^{\varphi(t)}_v}e^{-\frac{2Kv}{\sqrt{\varphi(t)}}}dvds
\end{align*}
and recalling the definition $H(t)$ we conclude \eqref{eq:differentialEquation}. Finally, \eqref{eq:solutionODE} comes as the solution of a classical first order linear ODE.
\end{proof}

\begin{proof}[Proof of Lemma \ref{lem:finitenessPhi}]
We estimate the terms in $H(t)$, see \eqref{eq:H(t)}. Note that $H(t)$ can be rewritten for $t>1$ with $g(t)\vee 1=g(t)$
\[H(t)=e^{-\frac{2Kt}{\sqrt{g(t)}}}\Pbb{\Oc^{g(t)}_t}-\frac{4K^2}{g(t)}\int_{0}^{t}\int_{0}^{s}\Pbb{\Oc^{g(t)}_v}e^{-\frac{2Kv}{\sqrt{g(t)}}}dvds.\]
Estimating $\Pbb{\Oc^{g(t)}_t}\leq \Pbb{\Oc_t}$, $e^{-\frac{2Kv}{\sqrt{g(t)}}}\leq 1$ and using the fact that $\Phi(t)$ is non-decreasing we arrive at
\begin{align}\label{eq:H(t)secondTerm}
\nonumber\frac{4K^2}{g(t)}\int_{0}^{t}\int_{0}^{s}\Pbb{\Oc^{g(t)}_v}e^{-\frac{2Kv}{\sqrt{g(t)}}}dvds&\leq \frac{4K^2t}{g(t)}\Phi(t)=\frac{4K^2t}{\gN{t}}\frac{\Phi(t)}{\gN{t}}\\
&=o\lbrb{\frac{\Phi(t)}{\gN{t}}}.
\end{align}
For the last line we use \eqref{eq:mildCondition}. Therefore, we need to discuss the first term of $H(t)$ only.

Denote by $g_1(t):=g(t)/\ln(t)$, for $t>2$. Distinguishing upon the times of $\Delta^{g_1(t)}_1,\Delta^{\theta g(t)}_1 $,  for some $\theta<1$, we get
\begin{align*}
&\Pbb{\Oc^{g(t)}_t}\leq \Pbb{\Oc^{g(t)}_t;\,\Delta^{\theta g(t)}_1\leq t}+\Pbb{\Oc^{g(t)}_t;\,\Delta^{g_1(t)}_1\leq t,\Delta^{\theta g(t)}_1>t}+\Pbb{\Oc^{g(t)}_t;\,\Delta^{g_1(t)}_1>t}.
\end{align*}
Note that since we work with the truncated subordinator and the corresponding event $\Oc^{g(t)}_t$, $\Delta^{a}_1\sim Exp\lbrb{2K/\sqrt{a}-2K/\sqrt{g(t)}}$, for $a<g(t)$. Note that always one can in a crude manner estimate the derivative
\[\Pbb{\Delta^{a}_1\in ds}\leq \lbrb{\frac{2K}{\sqrt{a}}-\frac{2K}{\sqrt{g(t)}}}ds\]
something will use extensively but implicitly below.

We note that from Lemma \ref{lem:Markov} -- which is stated and proven below -- we have that for any $c>0$ and fixed $n\in\Nb^+$ with $\delta=1$ the following inequality holds
\[\Pbb{\tau^{g_1(t)}_t>cg(t)}\leq e^{K\sqrt{nc^{-1}}\frac{t\ln^{1/2}(t)}{\sqrt{g(t)}}\int_{0}^{n/c}\lbrb{e^{s}-1}\frac{ds}{s^{3/2}}}t^{-n}\lesssim t^{-n}\]
since $\liminf\ttinf{t}g(t)/t^{2}\ln(t)=\infty$, i.e. \eqref{eq:mildCondition} holds. Therefore since $\Phi(t)=\int_{0}^{t}\Pbb{\Oc_s}ds$ is non-decreasing
\begin{align}\label{eq:Lemma o()-1}
\nonumber&\Pbb{\Oc^{g(t)}_t;\,\Delta^{g_1(t)}_1>t}=\Pbb{\Oc^{g_1(t)}_t}\Pbb{\Delta^{g_1(t)}_1>t}\leq\Pbb{\Oc^{g_1(t)}_t}\leq\\
& \Pbb{\tau^{g_1(t)}_t>g(t)}\lesssim \frac{1}{t^n}\lesssim
\frac{\int_{0}^{t}\Pbb{\Oc_s}ds}{t^{n-1}}=\frac{\Phi(t)}{t^{n-1}}.
\end{align}
Similarly disintegrating the time of arrival of $\Delta^{g_1(t)}_1$ and using that the maximal jump does not exceed $\theta g(t)$ we derive that
\begin{align*}
\nonumber& \Pbb{\Oc^{g(t)}_t;\,\Delta^{g_1(t)}_1\leq t,\Delta^{\theta g(t)}_1>t}=\int_{s=0}^{t}\Pbb{\Oc^{g(t)}_t;\Delta^{g_1(t)}_1\in ds,\Delta^{\theta g(t)}_1>t}=\\
\nonumber&\int_{s=0}^{t}\Pbb{\Oc^{g(t)}_t;\Delta^{\theta g(t)}_1>t|\Delta^{g_1(t)}_1=s}\Pbb{\Delta^{g_1(t)}_1\in ds}\leq\\ 
\nonumber&\int_{s=0}^{t}\Pbb{\Oc^{g_1(t)}_s;\tau_t^{g(t)}>g(t);\Delta^{\theta g(t)}_1>t|\Delta^{g_1(t)}_1=s}\Pbb{\Delta^{g_1(t)}_1\in ds}
\end{align*}
Estimating the density $\Pbb{\Delta^{g_1(t)}_1\in ds}$ and using that conditionally on $\lbcurlyrbcurly{\Delta^{g_1(t)}_1=s}$  the process $\tau^{g(t)}$ runs as $\tau^{g_1(t)}$ until time $s$ at which time it makes a jump $\tau_s-\tau_{s-}\in \lbrb{g_1(t),\theta g(t)}$ because of $\lbcurlyrbcurly{\Delta^{\theta g(t)}_1>t}$. Then
\[\tau^{g(t)}_t\stackrel{d}=\tau^{g_1(t)}_{s-}+\tau^{g(t)}_s-\tau^{g(t)}_{s-}+\tilde{\tau}^{g(t)}_{t-s},\]
where $\tilde{\tau}^{g(t)}$ is a copy of $\tau^{g(t)}$ independent of $\tau^{g_1(t)}_{s-}$. Therefore continuing the estimates above we get that
\begin{align}\label{eq:Lemma o()-2}
\nonumber& \Pbb{\Oc^{g(t)}_t;\,\Delta^{g_1(t)}_1\leq t,\Delta^{\theta g(t)}_1>t}\leq\\
\nonumber&\frac{2K\ln^{1/2}(t)}{\sqrt{g(t)}}\int_{0}^{t}\Pbb{\Oc^{g_1(t)}_s;\tilde{\tau}^{g(t)}_{t-s}+\tau^{g_1(t)}_{s-}>(1-\theta) g(t)}ds=\\
\nonumber&\frac{2K\ln^{1/2}(t)}{\sqrt{g(t)}}\int_{0}^{t}\Pbb{\Oc^{g_1(t)}_s;\tau^{g_1(t)}_{s-}>(1-\theta) g(t)/2,\tilde{\tau}^{g(t)}_{t-s}>(1-\theta) g(t)-\tau^{g_1(t)}_{s-}}ds+\\
\nonumber&\frac{2K\ln^{1/2}(t)}{\sqrt{g(t)}}\int_{0}^{t}\Pbb{\Oc^{g_1(t)}_s;\tau^{g_1(t)}_{s-}\leq (1-\theta) g(t)/2,\tilde{\tau}^{g(t)}_{t-s}>(1-\theta) g(t)-\tau^{g_1(t)}_{s-}}ds\leq\\
\nonumber &\frac{2K\ln^{1/2}(t)}{\sqrt{g(t)}}\int_{0}^{t}\Pbb{\Oc^{g_1(t)}_s;\tau^{g_1(t)}_{s-}>(1-\theta) g(t)/2}ds+\\
&\frac{2K\ln^{1/2}(t)}{\sqrt{g(t)}}\int_{0}^{t}\Pbb{\Oc^{g_1(t)}_s;\tilde{\tau}^{g(t)}_{t-s}>(1-\theta) g(t)/2}ds.
\end{align}
We estimate the last two terms in \eqref{eq:Lemma o()-2}. Due to the independence of $\tilde{\tau}^{g(t)}_{t-s}$ of $\Oc^{g_1(t)}_s$, the fact that $\tilde\tau^{g(t)}$ is a copy of $\tau^{g(t)}$, the fact that $\tau$ is a stable subordinator with index $1/2$ and \eqref{eq:lawTau} which describes the law of $\tau_t$ we get the trivial upper bounds
\begin{align*}
&\frac{2K\ln^{1/2}(t)}{\sqrt{g(t)}}\int_{0}^{t}\Pbb{\Oc^{g_1(t)}_s;\tilde{\tau}^{g(t)}_{t-s}>(1-\theta) g(t)/2}ds\leq\\ &\frac{2K\ln^{1/2}(t)}{\sqrt{g(t)}}\Pbb{\tau_t\geq (1-\theta)g(t)/2}\int_{0}^{t}\Pbb{\Oc^{g_1(t)}_s}ds\leq\frac{Ct\ln^{1/2}(t)}{\sqrt{(1-\theta)/2}g(t)}\Phi(t), 
\end{align*}
for some absolute constant $C>0$. The other term we estimate as before in \eqref{eq:Lemma o()-1}
\begin{align*}
&\frac{2K\ln^{1/2}(t)}{\sqrt{g(t)}}\int_{0}^{t}\Pbb{\Oc^{g_1(t)}_s;\tau^{g_1(t)}_{s-}>(1-\theta) g(t)/2}ds\leq\\ &\frac{2Kt\ln^{1/2}(t)}{\sqrt{g(t)}}\Pbb{\tau^{g_1(t)}_t>(1-\theta) g(t)/2}\leq C(\theta)\frac{\Phi(t)}{t^{n-1}}.
\end{align*}
Therefore collecting the terms above we get
\begin{equation}\label{eq:afterError1}
\Pbb{\Oc^{g(t)}_t;\,\Delta^{g_1(t)}_1\leq t,\Delta^{\theta g(t)}_1>t}\leq C_1(\theta)o(1)\frac{\Phi(t)}{\sqrt{g(t)}},
\end{equation}
since \eqref{eq:mildCondition} holds and $n$ can be chosen as big as we wish.

Finally consider the case when $\Delta^{\theta g(t)}_1\leq t$. Then  estimating \[\Pbb{\Delta^{\theta g(t)}_1\in ds}=\frac{2K}{\gN{t}}\lbrb{\frac{1}{\sqrt{\theta}}-1}e^{-\lbrb{\frac{2K}{\gN{t}}\lbrb{\frac{1}{\sqrt{\theta}}-1}s}}ds\leq \frac{2K}{\gN{t}}\lbrb{\frac{1}{\sqrt{\theta}}-1}ds\]
we get that
\begin{align*}
&\Pbb{\Oc^{g(t)}_t,\Delta^{\theta g(t)}_1\leq t}=\int_{0}^{t}\Pbb{\Oc^{g(t)}_t,\Delta^{\theta g(t)}_1\in ds}\leq\\
&\frac{2K}{\gN{t}}\lbrb{\frac{1}{\sqrt{\theta}}-1}\int_{0}^{t}\Pbb{\Oc^{g(t)}_s}ds\leq \frac{2K}{\gN{t}}\lbrb{\frac{1}{\sqrt{\theta}}-1}\Phi(t).
\end{align*}
Collecting this term, \eqref{eq:afterError1}, \eqref{eq:Lemma o()-1} we get that
\[\limsup\ttinf{t}\frac{\Pbb{\Oc^{g(t)}_t}\gN{t}}{2K\Phi(t)}\leq \lbrb{\frac{1}{\sqrt{\theta}}-1}\]

Setting $\theta\to 1$ we conclude the statement that $H(t)=o\lbrb{\frac{\Phi(t)}{\gN{t}}}$ and $\rho(t)=o\lbrb{\frac{1}{\gN{t}}}$. Then this allows together with 
\[\Phi(t)=\Phi(1)e^{\int_{1}^{t}\frac{2K}{\gN{s}}ds+\int_{1}^{t}\rho(s)ds}\]
to deduce \eqref{eq:finitenessPhi}.
\end{proof}

\begin{proof}[Proof of Theorem \ref{thm:asymptoticFiniteCase} and \eqref{eq:asymptoticInFiniteCase-1} of Theorem \ref{thm:asymptoticInFiniteCase}]
The claims that $\Pbb{\Oc_t}\sim \frac{2K\Phi(t)}{\gN{t}}$ in both theorems follows from the differential equation \eqref{eq:differentialEquation} and using the fact that
is immediate from Lemma \ref{lem:finitenessPhi} which establishes that $H(t)=o\lbrb{\frac{\Phi(t)}{\gN{t}}}$.  The behaviour of $\ln\lbrb{\Phi(t)}$ in \eqref{eq:asymptoticInFiniteCase-1} follows from \eqref{eq:solutionODE} and the fact that Lemma \ref{lem:finitenessPhi} shows that $\rho(t)=o\lbrb{\frac{1}{\gN{t}}}$ .

The claim $\Pbb{\Oc_t}\sim \Pbb{\Delta_1\leq t;\,\Oc_t}$ follows from \eqref{eq:equivWithJumps} as the second term is proved to be $o\lbrb{\frac{\Phi(t)}{\gN{t}}}$ and therefore $o\lbrb{\Pbb{\Oc_t}}$.
\end{proof}

\begin{proof}[Proof of \eqref{eq:asymptoticInFiniteCase-2} of Theorem \ref{thm:asymptoticInFiniteCase}]
The proof is immediate from Lemma \ref{lem:uniformIntegrability} below with $h=y=0$ which is the classical case.
\end{proof}

The next lemma proves a stronger claim than \eqref{eq:solutionODE} of Theorem \ref{thm:asymptoticInFiniteCase} as it provides some form of uniformity. For any $y\geq 0$ and $h\geq 0$ define $g^h_y(s):=\lbrb{g(s+h)-y}\vee 1$ and $g_{y,h}(s):=\lbrb{g(s+h)-y}$ and define $\Oc_t(h,y)=\{\tau_s>g_{y,h}(s),s\leq t\}$ and 
\[\Oc^{g^h_y}_t:=\Oc^{g^h_y}_t(h,y)=\{\tau^{g^h_y(t)}_s>g_{y,h}(s),s\leq t\}.\] 
We note that
\begin{equation}\label{eq:relationO}
  \Pbb{\Oc_t(h,y);\Delta^{g^h_y(t)}_1>t}\leq \Pbb{\Oc^{g^h_y}_t}\leq \Pbb{\Oc_t(h,y)}.
\end{equation}
We denote as well
\[\Phi^h_y(t)=\int_{0}^{t}\Pbb{\Oc_t(h,y)}ds;\,\,\phi^{h}_y(t)=\lbrb{\Phi^h_y(t)}'\]
and consider the more general differential equation with $H^h_y$ defined as in \eqref{eq:H(t)} but with the functions $g^h_y$ and $g_{y,h}$
\begin{equation}\label{eq:generalODE}
\phi^{h}_y(t)-\frac{2K}{\sqrt{g^h_y(t)}}\Phi^h_y(t)=H^h_y(t).
\end{equation}

Finally denote by 
\begin{equation}\label{eq:rho_h_y}
\rho^h_y(t):=\frac{H^h_y(t)}{\Phi^h_y(t)}.
\end{equation}

We have the following claim
\begin{lemma}\label{lem:uniformIntegrability}
Let $f$ satisfy the usual conditions \eqref{eq:growthCondition} and $\liminf\ttinf{t}g(t)/t^2\ln^{8/5+\epsilon}(t)=\infty$ for some $\epsilon>0$, i.e. \eqref{eq:conditionGrowth} holds. Then, for any $A>3$, $h> 0$,  $y>g(h)\vee 0$ and any $t>f(Ay)\vee t(A)$, where $t(A)$ satisfies the equation $g(t(A))=1+2A^{-1}$ we have the following bounds 
\begin{equation}\label{eq:uniformHestimate}
\rho^h_y(t)\leq C(A)\frac{1}{t\ln^{1+\frac{\epsilon}{2}}(t)}\lbrb{1+\frac{1}{\lbrb{f(y)-h}}}.
\end{equation}
There exists $u(t)\to 0$, as $t\to\infty$ such that for all $h,y$ with constraints above we have that
\begin{equation}\label{eq:uniformComparisonToG}
\rho^h_y(t)\leq \frac{u(t)}{\sqrt{g(t)}}\lbrb{1+\frac{1}{\lbrb{f(y)-h}}}.
\end{equation}
The last estimate \eqref{eq:uniformComparisonToG} holds with $\sqrt{g^h_y(t)}$ instead of $\sqrt{g(t)}$.
The estimates \eqref{eq:uniformHestimate} and \eqref{eq:uniformComparisonToG} also hold for $h=y=0$ without the factor $1/\lbrb{f(y)-h}$.
\end{lemma}
\begin{remark}
 The fact that $y>g(h)$ is to ensure that $g(h)-y<0$ since for small times $s$ the subordinator cannot cross immediately a positive boundary which will be the case if $g(h)>y$ and we will be dealing with trivialities like $\Pbb{\Oc_s(y,h)}=0$.
\end{remark}
\begin{proof}[Proof of Lemma \ref{lem:uniformIntegrability}]
	The case when $y=h=0$ can be dealt with as below with the only simplification that since $\Phi(t)>0$ we do not need \eqref{eq:lowerPhi} to introduce the function $\Phi(t)$ in the inequalities \eqref{eq:case1Final},\eqref{eq:Case 2A} and \eqref{eq:Case 2Cc-2b}. So we deal only with the uniform estimates in $h,y$ under the conditions of the lemma.
Applying \eqref{eq:relationO} for the first term of \eqref{eq:H(t)} taken with $\varphi(t)=g^h_y(t)$ we get that 
\begin{equation}\label{eq:H-1}
H^h_y(t)\leq\Pbb{\Oc^{g^h_y}_t}-\frac{4K^2}{g^h_y(t)}\int_{0}^{t}\int_{0}^{s}\Pbb{\tau_u>g_{y,h}(u),u\leq v\Big|\Delta^{g^h_y(t)}_1=v}e^{-\frac{2Kv}{\sqrt{g^h_y(t)}}}dvds.
\end{equation}
We work with $t>f(Ay)\vee t(A)$. Clearly the second term is bounded by
\begin{align*}
&\frac{4K^2}{g^h_y(t)}\int_{0}^{t}\int_{0}^{s}\Pbb{\tau_u>g_{y,h}(u),u\leq v\Big|\Delta^{g^h_y(t)}_1=v}e^{-\frac{2Kv}{\sqrt{g^h_y(t)}}}dvds=\\
&\frac{4K^2}{g^h_y(t)}\int_{0}^{t}\int_{0}^{s}\Pbb{\Oc^{g^h_y}_v}e^{-\frac{2Kv}{\sqrt{g^h_y(t)}}}dvds\leq\frac{4K^2 t}{g^h_y(t)}\Phi^h_y(t)\leq \tilde{B}(A)\frac{t}{g(t)}\Phi^h_y(t),
\end{align*}
where $\tilde{B}(A)=\frac{A}{A-1}$. For the second inequality we have used that, for any $y>g(h)\vee 0$ and $t>f(Ay)\vee t(A)$, we have that $g(t)=g(t)\vee 1$ and then with $B(A)=1-1/A$
\begin{align}\label{eq:estimateG}
\nonumber &g^{h}_y(t)\geq g(t)\lbrb{\frac{g(t+h)}{g(t)}-\frac{y}{g(t)}}\geq g(t)\lbrb{1-\frac{y}{g\lbrb{f(Ay)}}}\geq  B(A)g(t)\\
&g_{y,h}(t)\geq g(t)\lbrb{\frac{g(t+h)}{g(t)}-\frac{y}{g(t)}}g(t)\geq\lbrb{1-\frac{y}{g\lbrb{f(Ay)}}}\geq B(A)g(t).
\end{align}
However, since $\liminf\ttinf{t}g(t)/t^2\ln^{8/5+\epsilon}(t)=\infty$ we have that $t/g(t)=o\lbrb{\frac{1}{t\ln^{\frac{8}{5}+\epsilon}\lbrb{t}}}$ for $t>f(Ay)\vee t(A)$ and since $\tilde{B}(A)\frac{t}{g(t)}$ does not depend on $h$ and $y>g(h)\vee 0$ we see that we need to consider only $\Pbb{\Oc^{g^h_y}_t}$ for the proof of both \eqref{eq:uniformHestimate} and \eqref{eq:uniformComparisonToG} of our Lemma \ref{lem:uniformIntegrability}.

For brevity we put $w:=g^h_y$, $\tilde w:=g_{y,h}$,  and we denote by $\Delta^{w_1(t)}_k=\inf\{s>\Delta^{w_1(t)}_{k-1}:\tau^{w(t)}_s-\tau^{w(t)}_{s-}>w_1(t)\},\,\Delta^{w_1(t)}_0=0$, where we put $w_\delta(t):=w(t)/\ln^{\delta}(t)$. Put for the duration of this proof $\Oc^{g^h_y}_t:=\Oc^{w(t)}_t$. With the choice of $t>f(Ay)\vee t(A)$ we get that $w(t)=\tilde{w}(t),$ because when $y>\frac{1}{A-1}$ we have that 
\[g(t+h)-y\geq g(f(Ay))-y=\lbrb{A-1}y\geq 1\]
and otherwise 
\[g(t+h)-y\geq g(t(A))-y\geq 1+\frac{2}{A}-\frac{1}{A-1}>1\]
holds for $A>2$.

 To estimate $\Pbb{\Oc^{w(t)}_t}$ precisely we consider gradually several cases which correspond to different scenarios. Collecting all the estimates from each case will lead to our result.

\textbf{Case 1: $\Pbb{\Oc^{w(t)}_t;\Delta^{w_1(t)}_1>t\,}$}

We note that from Lemma \ref{lem:Markov} with $\delta=1$ and $c=1$ we get that for any $n\in\Nb^+$,
\begin{align*}
&\Pbb{\Oc^{w(t)}_t;\Delta^{w_1(t)}_1>t\,}\leq \Pbb{\tau^{w_1(t)}_t> w(t)}\leq  e^{tK\sqrt{n}\frac{\ln^{1/2}(w(t))}{\sqrt{ w(t)}}\int_{0}^{n}\lbrb{e^s-1}\frac{ds}{s^{3/2}}}e^{-n\ln(w(t))}.
\end{align*}
Therefore using \eqref{eq:estimateG} we are able to deduce that for $h>0$ and $y>g(h)\vee 0$ for any $t>f(Ay)\vee t(A)$, for each $n\in \Nb^+$,
\begin{align*}
& \Pbb{\tau^{w_1(t)}_t>\tilde w(t)}\leq  e^{2\tilde{B}(A)K\sqrt{n}\frac{t\ln^{1/2}(w(t))}{\sqrt{g(t)}}\int_{0}^{n}\lbrb{e^s-1}\frac{ds}{s^{3/2}}}e^{-n\ln(w(t))}.
\end{align*}
However, since $\liminf\ttinf{t}g(t)/t^2\ln^{8/5+\epsilon}(t)=\infty$ we see that $$2\tilde{B}(A)K\sqrt{n}\frac{t\ln^{1/2}(w(t))}{\sqrt{g(t)}}\int_{0}^{n}\lbrb{e^s-1}\frac{ds}{s^{3/2}}\leq C_n(A)\ln^{1/2}(w(t)),$$
where $C_n(A)>0$ depends solely on $n\in\Nb^+,A>3$. Henceforth, we get that
\begin{align*}
& \Pbb{\tau^{w_1(t)}_t>\tilde w(t)}\leq  e^{C_n(A)\ln^{1/2}(w(t))-\ln\lbrb{w(t)}}e^{-(n-1)\ln(w(t))}\leq \frac{C(n,A)}{w^{n-1}(t)}\leq \frac{C\lbrb{n,A}}{g^{n-1}(t)},
\end{align*}
where the last inequality follows from \eqref{eq:estimateG} and $C(n,A)>0$ from now on is a generic constant depending on $n,A$.

We note that in general for $s\leq f(y)-h,y>g(h)\vee 0$ we have that $\Pbb{\Oc^{w(t)}_s}=1$ as the curve $g(v+h)-y\leq 0,\,v\leq s$ and hence we obtain that for $t>t(A)\vee f\lbrb{Ay}>f(y)-h$
\begin{equation}\label{eq:lowerPhi}
	\Phi^h_y\lbrb{t}\geq\int_{0}^{f(y)-h}\Pbb{\Oc^{w(t)}_s}ds=f(y)-h.
\end{equation}
From the latter we get that
\begin{equation}\label{eq:case1Final}
\Pbb{\Oc^{w(t)}_t;\Delta^{w_1(t)}_1>t\,}\leq C(n,A)\frac{\Phi^h_y(t)}{\lbrb{f(y)-h}g^{n-1}(t)}\text{\\}.
\end{equation}
\newline
\textbf{Case 2: $\Pbb{\Oc^{w(t)}_t;\Delta^{w_1(t)}_1\leq t\,}$ }

First choose $\varepsilon<1/4$ so that  $1-4\varepsilon >0$ and define 
\[\Delta^{\varepsilon w(t)}_{k}=\inf\{s>\Delta^{\varepsilon w(t)}_{k-1}:\,\tau^{w(t)}_s-\tau^{w(t)}_{s-}>\varepsilon w(t)\},\,\Delta^{\varepsilon w(t)}_{0}=0.\] Note that each difference $\Delta^{\varepsilon w(t)}_{k}-\Delta^{\varepsilon w(t)}_{k-1}\sim Exp\lbrb{2K\lbrb{1/\sqrt{\varepsilon w(t)}-1/\sqrt{w(t)}}}$ since the jumps are defined for the truncated subordinator $\tau^{w(t)}$ and they form an independent sequence of random variables.\\

\textbf{Case 2A: $\Pbb{\Oc^{w(t)}_t;\Delta^{w_1(t)}_1\leq t,\Delta^{w_1(t)}_4>t,\,\Delta^{\varepsilon w(t)}_{1}>t}$}

We observe that putting at most $3$ jumps of at most size $\varepsilon w(t)$ and conditioning on $\{\Delta^{\varepsilon w(t)}_{1}>t,\Delta^{w_1(t)}_4>t\}$ we get 
\begin{align}\label{eq:Case 2A}
\nonumber&\Pbb{\Oc^{w(t)}_t;\Delta^{w_1(t)}_1\leq t,\Delta^{w_1(t)}_4>t,\,\Delta^{\varepsilon w(t)}_{1}>t}=\\
\nonumber&\sum_{k=1}^{3}\int_{0}^{t} \Pbb{\Oc^{w(t)}_t;\Delta^{w_1(t)}_k\in ds, \Delta^{w_1(t)}_{k+1}>t,\,\Delta^{\varepsilon w(t)}_{1}>t}\leq\\ 
\nonumber&\sum_{k=1}^{3}\int_{0}^{t}\Pbb{\Oc^{w(t)}_s;\,\tau^{w_1(t)}_t>(1-k\varepsilon)w(t)}\Pbb{\Delta^{w_1(t)}_k\in ds, \Delta^{w_1(t)}_{k+1}>t,\,\Delta^{\varepsilon w(t)}_{1}>t}\leq\\
&3\Pbb{\tau^{w_1(t)}_t>(1-4\varepsilon) w(t)}\leq C(\varepsilon,A,n)\frac{\Phi^h_y(t)}{\lbrb{f(y)-h}g^{n-1}(t)},
\end{align}
where for the last inequality we have used the procedure leading to \eqref{eq:case1Final} where $C(\varepsilon,A,n)>0$ is a generic constant. Also we have used that subtracting $k$ jumps of size larger than $w_1(t)$ then conditionally on $\{\Delta^{w_1(t)}_{k+1}>t\}$ we have that $\tau^{w(t)}=\tau^{w_1(t)}$.\\

\textbf{Case 2B: $\Pbb{\Oc^{w(t)}_t;\,\Delta^{w_1(t)}_4\leq t,\,\Delta^{\varepsilon w(t)}_{1}>t}$}

Conditioning on $\Delta^{w_1(t)}_1$ we easily get
\begin{align}\label{eq:Case 2B}
\nonumber&\Pbb{\Oc^{w(t)}_t;\,\Delta^{w_1(t)}_4\leq t,\,\Delta^{\varepsilon w(t)}_{1}>t}=\int_{0}^{t}\Pbb{\Oc^{w(t)}_t;\Delta^{w_1(t)}_1\in ds,\,\Delta^{w_1(t)}_4\leq t,\,\Delta^{\varepsilon w(t)}_{1}>t}\leq \\
\nonumber&\int_{0}^{t}\Pbb{\Oc^{w_1(t)}_s}\Pbb{\Delta^{w_1(t)}_1\in ds,\,\Delta^{w_1(t)}_4\leq t}\leq\Pbb{\Delta^{w_1(t)}_3\leq t} \int_{0}^{t}\Pbb{\Oc^{w_1(t)}_s}\Pbb{\Delta^{w_1(t)}_1\in ds}\leq\\
\nonumber &\lbrb{\Pbb{\Delta^{w_1(t)}_1\leq t}}^3 \frac{2K\ln^{1/2}(t)}{\sqrt{w(t)}}\int_{0}^{t}\Pbb{\Oc^{w_1(t)}_s}ds\leq C(A)\frac{t^3\ln^{2}(t)}{w^2(t)}\Phi^h_y(t)\leq\\
& \tilde{C}(A)\frac{t^3\ln^{2}(t)}{g^2(t)}\Phi^h_y(t)=\tilde{C}(A)\frac{t^3\ln^{2}(t)}{g^{3/2}(t)}\frac{\Phi^h_y(t)}{\sqrt{g(t)}}\leq  \tilde{C}'(A)\frac{\Phi^h_y(t)}{t\ln^{6/5+2\epsilon}(t)},
\end{align}
where in the first inequality we excluded $\{\Delta^{\varepsilon w(t)}_{1}>t\}$ and estimated \[\Pbb{\Oc^{w(t)}_t\Big|\Delta^{w_1(t)}_1=s,\Delta^{w_1(t)}_4\leq t}\leq \Pbb{\Oc^{w_1(t)}_s};\]
next for the second inequality we enlarged the time for possible arrivals of jump $2,3,4$; for the third inequality we further allowed each jump $2,3,4$ to take $t$ amount of time to occur and estimated the density of $\Delta^{w_1(t)}_1\sim Exp\lbrb{2K\lbrb{1/\sqrt{w_1(t)}-1/\sqrt{w(t)}}}$ generously with $2K/\sqrt{w_1(t)}$; for the fourth we note that similarly
 \[\Pbb{\Delta^{w_1(t)}_1\leq t}=1-e^{-2Kt/\sqrt{w_1(t)}}\leq 2Kt/\sqrt{w_1(t)};\] 
 fifth we use \eqref{eq:estimateG} to bound the expressions with $w(t)$ uniformly with $g(t)$ and finally we recall that $\liminf\ttinf{t}g(t)/t^2\ln^{8/5+\epsilon}(t)=\infty$, i.e. \eqref{eq:conditionGrowth} holds.\\

\textbf{Case 2C: $\Pbb{\Oc^{w(t)}_t;\,\Delta^{\varepsilon w(t)}_{1}\leq t}$}

Define $p(t)=\ln^{-\gamma}(t)$, $p^*(t)=1-p(t)$ and $\gamma<3/5+\epsilon$ to be chosen later. Define similarly as before the sequence of jumps exceeding the level $p^*(t)w(t)$
\[\Delta^{p^*(t) w(t)}_{k}=\inf\{s>\Delta^{p^*(t)  w(t)}_{k-1}:\,\tau^{w(t)}_s-\tau^{w(t)}_{s-}\in \lbrb{p^*(t)w(t),w(t)}\},\Delta^{p^*(t) w(t)}_{0}=0,\]
where we recall that we already work with a subordinator whose jumps larger than $w(t)$ have been truncated. We have again
\[\Delta^{p^*(t) w(t)}_{k}-\Delta^{p^*(t) w(t)}_{k-1}\sim Exp\lbrb{2K\lbrb{\frac{1}{\sqrt{p^*(t)w(t)}}-\frac{1}{\sqrt{w(t)}}}},\]
wherefrom we get easily from \eqref{eq:estimateG}, $\liminf\ttinf{t}g(t)/t^2\ln^{8/5+\epsilon}(t)=\infty$ and $1-\sqrt{1-x}\leq x$ that 
\begin{equation}\label{eq:estimateDensityJumpCloseToBoundary}
\Pbb{\Delta^{p^*(t) w(t)}_{1}\in ds}\leq \frac{2Kp(t)}{\sqrt{p^*(t)w(t)}}ds\leq  \frac{C(A)}{\sqrt{g(t)}\ln^{\gamma}(t)}ds\leq\frac{C(A)}{t\ln^{4/5+\epsilon/2+\gamma}(t)}ds
\end{equation}
since $g(1)=1$ and $g(t(A))=1+\frac{2}{A}$ imply that $t(A)>1$ and henceforth $p^*(t)\geq p^*(t(A))=1-\ln^{-\gamma}(t(A))>0$. \\

\textbf{Case 2Ca: $\Pbb{\Oc^{w(t)}_t;\,\Delta^{\varepsilon w(t)}_{1}\leq t, \Delta^{p^*(t) w(t)}_{1}<t}$}

We estimate ignoring the event $\{\Delta^{\varepsilon w(t)}_{1}\leq t\}$, disintegrating on the possible position of $\Delta^{p^*(t) w(t)}_{1}$ and estimating conditionally on $\{\Delta^{p^*(t) w(t)}_{1}=s\}$ that
\[\Pbb{\Oc^{w(t)}_t\Big|\Delta^{p^*(t) w(t)}_{1}=s}\leq \Pbb{\Oc^{w(t)}_s}\]
to get the following chain of inequalities
\begin{align}\label{eq:Case 2Ca}
\nonumber&\Pbb{\Oc^{w(t)}_t;\,\Delta^{\epsilon w(t)}_{1}\leq t, \Delta^{p^*(t) w(t)}_{1}<t}\leq \int_{0}^{t}\Pbb{\Oc^{w(t)}_s}\Pbb{\Delta^{p^*(t) w(t)}_{1}\in ds}\leq \\
&\frac{C(A)}{\sqrt{g(t)}\ln^{\gamma}(t)}\Phi^h_y(t)\leq \frac{C(A)}{t\ln^{4/5+\epsilon/2+\gamma}(t)}\Phi^h_y(t),
\end{align}
where we also used \eqref{eq:estimateDensityJumpCloseToBoundary}.\\

\textbf{Case 2Cb: $\Pbb{\Oc^{w(t)}_t;\,\Delta^{\varepsilon w(t)}_{2}\leq t, \Delta^{p^*(t) w(t)}_{1}>t}$}

We start again by disintegrating the time of first jump $\Delta^{\varepsilon w(t)}_{1}$,  estimating as in \eqref{eq:estimateDensityJumpCloseToBoundary} employing the same techniques \eqref{eq:estimateG}, 
\[\Pbb{\Delta^{\varepsilon w(t)}_{1}\in ds}\leq \frac{C(A,\varepsilon)}{\gN{t}}ds;\,\,\Pbb{\Delta^{\varepsilon w(t)}_{1}\leq t}\leq \frac{C(A,\varepsilon)t}{\sqrt{g(t)}}\]
and using $\liminf\ttinf{t}g(t)/t^2\ln^{8/5+\epsilon}(t)=\infty$ to obtain in the same way as in \eqref{eq:Case 2B} the preliminary estimate
\begin{align}\label{eq:Case 2Cb}
\nonumber &\Pbb{\Oc^{w(t)}_t;\,\Delta^{\varepsilon w(t)}_{2}\leq t, \Delta^{p^*(t) w(t)}_{1}>t}\leq \frac{C^2(A,\varepsilon)t}{g(t)}\int_{0}^{t}\Pbb{\Oc^{w(t)}_s}ds\leq\\
&C^2(A,\varepsilon)\frac{t}{\sqrt{g(t)}}\frac{\Phi^h_y(t)}{\sqrt{g(t)}}\leq\frac{C^2(A,\varepsilon)}{t\ln^{8/5+\epsilon}(t)}\Phi^h_y(t).
\end{align}

\textbf{Case 2Cc: $\Pbb{\Oc^{w(t)}_t;\,\Delta^{\epsilon w(t)}_{1}\leq t,\,\Delta^{\epsilon w(t)}_{2}> t,\,\Delta^{p^*(t) w(t)}_{1}>t}$}

We again disintegrate the position of $\Delta^{\epsilon w(t)}_{1}$ in similar fashion, then substitute the highest possible value of the jump at time $\Delta^{\epsilon w(t)}_{1}$ namely $p^*(t) w(t)$ to get for the end point $\tau^{w(t)}_t>w(t)-p^*(t)w(t)=p(t)w(t)$ and then use \eqref{eq:estimateG} to derive the preliminary estimate
\begin{align*}
&\Pbb{\Oc^{w(t)}_t;\,\Delta^{\varepsilon w(t)}_{1}\leq t,\,\Delta^{\varepsilon w(t)}_{2}> t,\,\Delta^{p^*(t) w(t)}_{1}>t}\leq \\
&\frac{C(A,\varepsilon)}{\sqrt{g(t)}}\int_{0}^{t}\Pbb{\Oc^{\varepsilon w(t)}_s,\tau^{\varepsilon w(t)}_t>p(t)w(t)}ds\leq\\
&\frac{C(A,\varepsilon)}{\sqrt{g(t)}}\int_{0}^{t}\Pbb{\Oc^{\varepsilon w(t)}_s,\,\tau^{\varepsilon w(t)}_s\leq p(t)w(t)/2,\,\tau^{\varepsilon w(t)}_t>p(t)w(t)}ds+\\
&\frac{C(A,\varepsilon)}{\sqrt{g(t)}}\int_{0}^{t}\Pbb{\Oc^{\varepsilon w(t)}_s,\,\tau^{\varepsilon w(t)}_s> p(t)w(t)/2}ds=\\
& S(t)+S^*(t).
\end{align*}
Let us first estimate $S(t)$ to see that its implicit dependence on $y$ is irrelevant. We note that 
\[\{\Oc^{\varepsilon w(t)}_s;\,\tau^{\varepsilon w(t)}_s\leq p(t)w(t)/2,\,\tau^{\varepsilon w(t)}_t>p(t)w(t)\}\subseteq \{\Oc^{\varepsilon w(t)}_s;\,\tau^{\varepsilon w(t)}_{t}-\tau^{\varepsilon w(t)}_{s}>p(t)w(t)/2\}.\]
Clearly from the fact that $\tau^{\varepsilon w(t)}_{t}-\tau^{\varepsilon w(t)}_{s}$ is independent of $\Oc^{\varepsilon w(t)}_s$  and $\tau^{\varepsilon w(t)}_{t}$ is a.s. increasing we are able to imply that
\begin{align}\label{eq:Case 2Cc-1}
\nonumber&S(t)\leq C(A,\varepsilon)\frac{\Pbb{\tau^{\varepsilon w(t)}_{t}>p(t)w(t)/2}}{\sqrt{g(t)}}\int_{0}^{t}\Pbb{\Oc^{w(t)}_s}ds\leq\\
\nonumber&C(A,\varepsilon)\frac{\Pbb{\tau_{t}>p(t)w(t)/2}}{\sqrt{g(t)}}\int_{0}^{t}\Pbb{\Oc^{w(t)}_s}ds\leq D(A,\varepsilon)\frac{t}{g(t)\sqrt{p(t)}}\int_{0}^{t}\Pbb{\Oc^{w(t)}_s}ds\leq\\
&D(A,\varepsilon)\frac{t\ln^{\gamma/2}(t)}{\sqrt{g(t)}}\frac{\Phi^h_y(t)}{\sqrt{g(t)}}\leq D(A,\varepsilon)\frac{1}{t\ln^{8/5+\epsilon-\gamma/2}(t)}\Phi^h_y(t),
\end{align}
where we have first estimated $\tau^{\epsilon w(t)}_{t}\leq \tau_{t}$, then used that since $\tau_t$ is stable with index $1/2$ we have that 
\[\Pbb{\tau_{t}>p(t)w(t)/2}=\Pbb{\tau_{1}>t^{-2}p(t)w(t)/2}\leq D\frac{t}{\sqrt{p(t)w(t)}},\]
applied the definition of $p(t)=\ln^{-\gamma}(t)$ and \eqref{eq:estimateG} to compare uniformly $w(t)$ with $g(t)$ from below and the recurring $\liminf\ttinf{t}g(t)/t^2\ln^{8/5+\epsilon}(t)=\infty$.

Let us next estimate $S^*(t)$. Denote $\delta=1+\gamma$ and recall that by definition $w_\delta(t)=w(t)/\ln^{\delta}(t)$. Define as always $\Delta^{w_{\delta}(t)}_1$ the time of first jump exceeding $w_\delta(t)$ and note that its density can be estimated as in similar cases before with the help of \eqref{eq:estimateG} by 
\[\Pbb{\Delta^{w_{\delta}(t)}_1\in ds}\leq \frac{C(A)\ln^{\delta/2}(t)}{\sqrt{g(t)}}ds.\] 
We write each integrand of $S^*$ as follows
\begin{align*}
&\Pbb{\Oc^{\varepsilon w(t)}_s,\,\tau^{\varepsilon w(t)}_s> p(t)w(t)/2}=\\
&\Pbb{\Oc^{\varepsilon w(t)}_s,\,\tau^{\varepsilon w(t)}_s> p(t)w(t)/2,\Delta^{w_{\delta}(t)}_1\leq s}+\Pbb{\Oc^{\varepsilon w(t)}_s,\,\tau^{\varepsilon w(t)}_s> p(t)w(t)/2,\Delta^{w_{\delta}(t)}_1>s}.
\end{align*}
For the first we get
\begin{align}\label{eq:Case 2Cc-2a}
\nonumber& \frac{C(A,\varepsilon)}{\sqrt{g(t)}}\int_{0}^{t}\Pbb{\Oc^{\varepsilon w(t)}_s,\,\tau^{\varepsilon w(t)}_s> p(t)w(t)/2,\Delta^{w_{\delta}(t)}_1\leq s}ds=\\
\nonumber& \frac{C(A,\varepsilon)}{\sqrt{g(t)}}\int_{0}^{t}\int_{0}^{v}\Pbb{\Oc^{\varepsilon w(t)}_s,\,\tau^{\varepsilon w(t)}_s> p(t)w(t)/2,\Delta^{w_{\delta}(t)}_1\in dv }ds\leq \\
\nonumber&\frac{C(A,\varepsilon)}{\sqrt{g(t)}}\int_{0}^{t}\int_{0}^{v}\Pbb{\Oc^{\varepsilon w(t)}_v}\Pbb{\Delta^{w_{\delta}(t)}_1\in dv}ds\leq
\frac{C(A,\varepsilon)t\ln^{\delta/2}(t)}{g(t)}\Phi^h_y(t)\leq\\
& \frac{C(A,\varepsilon)t\ln^{\delta/2}(t)}{\sqrt{g(t)}}\frac{\Phi^h_y(t)}{\sqrt{g(t)}}\leq\frac{C'(A,\varepsilon)}{t\ln^{11/10+\epsilon-\gamma/2}(t)}\Phi^h_y(t),
\end{align}
where we have estimated as measures 
\begin{align*}
&\Pbb{\Oc^{\varepsilon w(t)}_s,\,\tau^{\varepsilon w(t)}_s> p(t)w(t)/2,\,\Delta^{w_{\delta}(t)}_1\in dv}\leq\\
& \Pbb{\Oc^{\varepsilon w(t)}_v,\,\Delta^{w_{\delta}(t)}_1\in dv}\leq  \Pbb{\Oc^{\varepsilon w(t)}_v}\Pbb{\Delta^{w_{\delta}(t)}_1\in dv}.
\end{align*}
For the second integrand we simply estimate in the following generous manner truncating all events and putting the largest values at the point $t$, i.e. $\tau^{w_{\delta}(t)}_t$,
\begin{align*}
&\Pbb{\Oc^{\varepsilon w(t)}_s,\,\tau^{\varepsilon w(t)}_s> p(t)w(t)/2,\Delta^{w_{\delta}(t)}_1>s}\leq \Pbb{\tau^{w_{\delta}(t)}_t>p(t)w(t)/2}.
\end{align*}
Using the exponential Markov inequality with $\lambda=2n w^{-1}_\delta(t)$, the last expression for the \LL-Khintchine exponent of $\tau^{w_\delta(t)}$ in \eqref{eq:LevyKhintchineTruncated},  $p(t)=\ln^{-\gamma}(t)$ and $\delta=1+\gamma$ we get that
\begin{align*}
&\Pbb{\tau^{w_{\delta}(t)}_t>p(t)w(t)/2}\leq e^{-\lambda p(t)w(t)/2}\mathbb{E}\bigl[e^{\lambda \tau^{w_{\delta}(t)}_t } \bigr] \\ &\quad =e^{tK\sqrt{2n}\frac{\ln^{\delta/2}(t)}{\sqrt{w(t)}}\int_{0}^{2n}\lbrb{e^{s\frac{2n}{w_{\delta}(t)}}-1}\frac{ds}{s^{3/2}}}e^{-2n w^{-1}_\delta(t) p(t)w(t)/2}=\\
&\quad = e^{tK\sqrt{2n}\frac{\ln^{\delta/2}(t)}{\sqrt{w(t)}}\int_{0}^{2n}\lbrb{e^{s\frac{2n}{w_{\delta}(t)}}-1}\frac{ds}{s^{3/2}}}e^{-n \ln(t)}\\
&\quad \leq
e^{C(A)K\sqrt{n}\frac{1}{\ln^{3/10-\gamma/2+\epsilon/2}(t)}\int_{0}^{2n}\lbrb{e^{s\frac{2n}{w_{\delta}(t)}}-1}\frac{ds}{s^{3/2}}}e^{-n \ln(t)},
\end{align*}
where for the exponent of the first factor in the last inequality we have used \eqref{eq:estimateG} and $\liminf\ttinf{t}g(t)/t^2\ln^{8/5+\epsilon}(t)=\infty$. Therefore
\begin{align}\label{eq:Case 2Cc-2b}
\nonumber&\frac{C(A,\varepsilon)}{\sqrt{g(t)}}\int_{0}^{t}\Pbb{\Oc^{\varepsilon w(t)}_s,\,\tau^{\varepsilon w(t)}_s> p(t)w(t)/2,\,\Delta^{w_{\delta}(t)}_1>s}ds\leq\\
\nonumber&t\frac{C(A,\varepsilon)}{\sqrt{g(t)}} \nonumber e^{C(A)K\sqrt{n}\frac{1}{\ln^{3/10-\gamma/2+\epsilon/2}(t)}}e^{-n \ln(t)}\frac{\Phi^h_y(t)}{f(y)-h}\leq \\
&\frac{C(A,\varepsilon)}{\sqrt{t^{n-1}g(t)}}\frac{\Phi^h_y(t)}{f(y)-h}
\end{align}
provided $\gamma=1/5$ as then the positive exponent is bounded since the inequality  $3/10-\gamma/2+\epsilon/2>0$ holds and $n/w_\delta(t)$ is bounded for $t>t(A)$. The appearance of the factor $\frac{\Phi^h_y(t)}{f(y)-h}$ follows from the inequality \eqref{eq:lowerPhi}.

We  collect all terms in \eqref{eq:Case 2Cc-2a}, \eqref{eq:Case 2Cc-2b}, \eqref{eq:Case 2Cc-1}, \eqref{eq:Case 2Cb}, \eqref{eq:Case 2Ca} updating for $\gamma=1/5$ and choosing $n=7$ to get
\begin{align}\label{eq:firstFinalEstimate}
&\Pbb{\Oc^{w(t)}_t;\,\Delta^{\varepsilon w(t)}_{1}\leq t}\leq \frac{C(A,\varepsilon,\gamma,n)}{t\ln^{1+\epsilon/2}(t)}\Phi^h_y(t)+t^{-4}\frac{\Phi^h_y(t)}{f(y)-h}.
\end{align}
We note the worst logarithmic bound comes from $\eqref{eq:Case 2Ca}$.

We are ready now to conclude the proof by noting by the same choice of $n=7$ all bounds in \eqref{eq:Case 2A}, \eqref{eq:Case 2B} and \eqref{eq:case1Final} are of at most the same and faster decay taking into account that $g(t)/t\to \infty,\to\infty$. Therefore we conclude that uniformly for $t>f(Ay)\vee t(A)$, $y>g(h)\vee 0$, we have that
\begin{align}\label{eq:FinalEstimate}
&\Pbb{\Oc^{w(t)}_t}\leq \frac{C(A,\epsilon,\gamma,n)}{t\ln^{1+\epsilon/2}(t)}\Phi^h_y(t)+t^{-4}\frac{\Phi^h_y(t)}{f(y)-h}.
\end{align}
and using \eqref{eq:H-1} we conclude that
\[H^h_y(t)\leq C(A,\varepsilon,7,1/5)\lbrb{\frac{1}{t\ln^{1+\epsilon/2}(t)}+\frac{1}{t^4(f(y)-h)}}\Phi^h_y(t)\]
and since the bound is uniform and we get that $\rho^h_y$ is integrable at infinity we deduce our proof of \eqref{eq:uniformHestimate}. To prove \eqref{eq:uniformComparisonToG} we note that all estimates above which contain $t^n$ can be uniformly majorized by $u_1(t)/(f(y)-h)\sqrt{g(t)}$ with $u_1(t)$ uniformly in $y$ tending to zero. For the other estimates \eqref{eq:Case 2B},  \eqref{eq:Case 2Ca}, \eqref{eq:Case 2Cb}, \eqref{eq:Case 2Cc-2a} and \eqref{eq:Case 2Cc-1} choosing the worst estimates we get that they do not exceed with $\gamma=1/5$
\[\lbrb{\lbrb{\frac{t\ln^{3/5}(t)}{\sqrt{g(t)}}}\bigvee\lbrb{\frac{1}{\ln^{1/5}(t)}}\bigvee\lbrb{\frac{t}{\gN{t}}}\bigvee\lbrb{\frac{t^3\ln^2(t)}{g^{3/2}(t)}}}\frac{\Phi^h_y(t)}{\sqrt{g(t)}}=u_2(t)\frac{\Phi^h_y(t)}{\sqrt{g(t)}}.\]
Therefore from \eqref{eq:conditionGrowth} we get $u_2(t)\to 0$ and henceforth we get
\[\rho^h_y(t)\leq C(A) \lbrb{\frac{u_1(t)}{(f(y)-h)\sqrt{g(t)}}+\frac{u_2(t)}{\sqrt{g(t)}}},\]
which due to uniformity in $y$ settles the last claim. We could easily observe that in each bound we obtained along the way we estimated $w(t)\geq B(A)g(t)$ in the denominator and then it easily follows that \eqref{eq:uniformComparisonToG} holds with $w=g_{y,h}$ for $g$.
\end{proof}
The next Lemma is auxiliary and is used throughout the proof above 
\begin{lemma}\label{lem:Markov}
Let $a>0$ then we have that with $a_\delta=a/\ln^\delta(a)$ and $\delta>0$ for any $t>0$, $c>0$  and $n\in \Nb$
\begin{equation}
\Pbb{\tau^{a_\delta}_t>ca}\leq e^{\frac{tK\sqrt{n}\ln^{\delta/2}(a)}{\sqrt{ca}}\int_{0}^{n/c}\lbrb{e^s-1}\frac{ds}{s^{3/2}}}e^{-n\ln^{\delta}(a)}
\end{equation}
\end{lemma}
\begin{proof}
This is a simple proof using the Markov inequality together with $\Pi(ds)=Kds/s^{3/2}$, \eqref{eq:LevyKhintchineTruncated} and a choice of $\lambda=\frac{n}{c}a^{-1}_\delta$.
\end{proof}

\section{Proofs for Section \ref{sec:transience}}
\begin{proof}[Proof of Theorem \ref{thm:transience}]
Since $\Ebb{f(\tau_1)}<\infty$ we have thanks to \eqref{eq:equivalenceToExpectation} that $J(g)<\infty$ and hence according to Lemma \ref{lem:finitenessPhi} that $\Phi(\infty)<\infty$. Therefore the clocks $\mathfrak{C}_t$ defined in \eqref{eq:clocks} converge in distribution to $\mathfrak{C}$.

Next we show that under any possible limit of $\Pb_t$ the inverse local time at zero $\tau=\lbcurlyrbcurly{\tau_s}_{s\geq 0}$ satisfies $\Qbb{\tau_x<\infty}=\Pbb{\mathfrak{C}>x}$ and $\Qbb{\tau_x\in dy;\,\Bc}=\frac{\Phi^x_y(\infty)}{\Phi(\infty)}\Pbb{\tau_x\in dy;\Oc_x},$ for any $\Bc\subseteq\Oc_h;\,\Bc\in\Fc_h$, where $\Qb$ denotes a generic possible weak limit. Thus the possible limit of $\tau$ under $\Pb_t$ is unique. Note that for any $x>0$, $t>x$ , $y>g(x)$ and $\Bc\subseteq\Oc_x;\,\Bc\in\Fc_x$
\begin{align*}
\Pbb{\tau_x\in dy; \Bc|\Oc_t}=\frac{\Pbb{\Oc_{t-x}(x,y)}}{\Pbb{\Oc_t}}\Pbb{\tau_x\in dy;\,\Bc;\,\Oc_x},
\end{align*}
where $\Oc_{t}(x,y)=\{\tau_{s}>g(s+x)-y;\,\forall s\leq t\}$. Clearly, from  Theorem \ref{thm:asymptoticFiniteCase} we have that with $\Phi^x_y(t)=\int_{0}^{t}\Pbb{\Oc_{s}(x,y)}ds$ and $x,y$ fixed
\[\lim\ttinf{t}\frac{\Pbb{\Oc_{t-x}(x,y)}}{\Pbb{\Oc_t}}=\frac{\Phi^x_y(\infty)}{\Phi(\infty)}\lim\ttinf{t}\frac{\sqrt{g(t)}}{\sqrt{g(t-x+x)-y}}=\frac{\Phi^x_y(\infty)}{\Phi(\infty)},\]
where $\Phi^x_y(\infty)<\infty$ follows from Theorem \ref{thm:asymptoticFiniteCase} since $g(s+x)-y$ satisfies \eqref{eq:mildCondition} and \eqref{eq:finitenessPhi}, namely 
\[\int_{1}^{\infty}\frac{ds}{\sqrt{g(s+x)-y}}<\infty.\]
This shows the convergence of $\tau$ to a unique limit in law under $\Pb_t$ and shows that
\[\Qbb{\tau_x\in dy}=\frac{\Phi^x_y(\infty)}{\Phi(\infty)}\Pbb{\tau_x\in dy;\Oc_x},\]
which proves \eqref{eq:tauTransience}. Then \eqref{eq:h-transformTransience} follows immediately.
Since $\Pbb{\Oc_{t}(x,y)}$ is monotone in $y$, for any $g(x)<y<B$ with $B>0$, we can use the dominated convergence theorem to get using the definition of $\Phi^x_y$
\begin{align*}
\Qbb{\tau_x\in(g(x),B)}&=\lim\ttinf{t}\int_{y=g(x)}^{B}\Pbb{\tau_x\in dy\mid \Oc_t}\\
&=\lim\ttinf{t}\int_{y=g(x)}^{B}\frac{\Pbb{\Oc_{t-x}(x,y)}}{\Pbb{\Oc_t}}\Pbb{\tau_x\in dy;\Oc_x}\\
&=\frac{1}{\Phi(\infty)}\int_{y=g(x)}^{B}\Phi^x_y(\infty)\Pbb{\tau_x\in dy;\Oc_x}\\
&=\frac{1}{\Phi(\infty)}\int_{y=g(x)}^{B}\int_{0}^{\infty}\Pbb{\Oc_v(x,y)}
dv\Pbb{\tau_x\in dy;\Oc_x}\\
&=\frac{\int_{x}^{\infty}\Pbb{\Oc_u; \tau_x\in (g(x),B)}du}{\Phi(\infty)}.
\end{align*}
Using the monotone convergence theorem we get that
\begin{align*}
&\Qbb{\tau_x<\infty}=\lim\ttinf{B}\frac{\int_{x}^{\infty}\Pbb{\Oc_u; \tau_x\in (g(x),B)}du}{\Phi(\infty)}=\frac{\int_{x}^{\infty}\Pbb{\Oc_u}du}{\Phi(\infty)}=\Pbb{\mathfrak{C}>x}.
\end{align*}
Since $\tau$ is a.s. non-decreasing we conclude that under $\Qb$ there is a random explosion time $\Tc$ for $\tau$ such that $\tau_s=\infty,s\geq \Tc$ and $\Tc\stackrel{d}{=}\mathfrak{C}$.

Next we prove that the random elements $Z^t=\lbcurlyrbcurly{\tau,\Delta^{g(t)}_1}\in\Dc\lbrb{0,\infty}\times \Rb^+$ converge under $\Pb_t$ to a random element $Z$ and we specify the structure of $Z$ under $\Qb$.  
For any $x>0$, $y>g(x)$, $t>b>a>x$ with $a,b$ fixed and $t\to\infty$ we have that with $\Bc\subseteq\Oc_x;\,\Bc\in\Fc_x,$ 
\begin{align*}
&\Pbb{\tau_x\in dy;\,\Bc;\,\Delta^{g(t)}_1\in (a,b);\,\Oc_t}=\\
&\int_{a}^{b}\Pbb{\tau_x\in dy;\,\Bc;\,\Delta^{g(t)}_1\in ds;\,\Oc_x}=\int_{a}^{b}\Pbb{\tau_x\in dy;\,\Delta^{g(t)}_1\in ds;\,\Bc;\,\Oc_{\Delta^{g(t)}_1}}=\\
&\int_{a}^{b}\Pbb{\tau^{g(t)}_x\in dy;\,\Bc;\,\Oc^{g(t)}_s}\Pbb{\Delta^{g(t)}_1\in ds}\sim \frac{2K}{\gN{t}}\int_{a}^{b}\Pbb{\tau^{g(t)}_x\in dy;\,\Bc;\,\Oc^{g(t)}_s}ds\sim\\
&\frac{\Pbb{\Oc_t}}{\Phi(\infty)}\int_{a}^{b}\Pbb{\tau^{g(t)}_x\in dy;\,\Bc;\,\Oc^{g(t)}_s}ds=\frac{\Pbb{\Oc_t}}{\Phi(\infty)}\int_{a}^{b}\Pbb{\tau^{g(t)}_x\in dy;\,\Bc\mid\Oc^{g(t)}_s}\Pbb{\Oc^{g(t)}_s}ds,
\end{align*}
where we have used that since $\Delta^{g(t)}_1\sim Exp\lbrb{\frac{2K}{\gN{t}}}$ then for $s\in\lbbrbb{a,b}$ we have uniformly
\[\Pbb{\Delta^{g(t)}_1\in ds}=\frac{2K}{\gN{t}}e^{-\frac{2Ks}{\gN{t}}}ds\sim\frac{2K}{\gN{t}}ds \]
and then thanks to \eqref{eq:asymptoticFiniteCase} of Theorem \ref{thm:asymptoticFiniteCase} that $\frac{2K}{\gN{t}}\sim \frac{\Pbb{\Oc_t}}{\Phi(\infty)}$. Furthermore, since by definition $\frac{\Pbb{\Oc_s}}{\Phi(\infty)}ds=\Pbb{\mathfrak{C}\in ds}$, see \eqref{eq:clocks} in the limit, we continue the relations
\begin{align*}
&\Pbb{\tau_x\in dy;\,\Bc;\,\Delta^{g(t)}_1\in (a,b);\,\Oc_t}\sim \Pbb{\Oc_t}\int_{a}^{b}\Pbb{\tau^{g(t)}_x\in dy;\,\Bc\mid\Oc^{g(t)}_s}\frac{\Pbb{\Oc_s}}{\Phi(\infty)}ds=\\
&\Pbb{\Oc_t}\int_{a}^{b}\Pbb{\tau^{g(t)}_x\in dy;\,\Bc\mid\Oc^{g(t)}_s}\Pbb{\mathfrak{C}\in ds}\sim
\Pbb{\Oc_t}\int_{a}^{b}\Pbb{\tau_x\in dy;\,\Bc\mid\Oc_s}\Pbb{\mathfrak{C}\in ds}.
\end{align*}
Clearly, taking limits after conditioning on $\Oc_t$, we get that under $\Qb$, $\lbcurlyrbcurly{\lbcurlyrbcurly{\tau_s}_{s\leq \Tc};\Tc}$ has the law 
\begin{align}\label{eq:lawJoint}
 &\Qbb{\tau\in\Bc;\Tc\in\lbrb{a,b}}=\\
 \nonumber&\int_{a}^{b}\Qbb{\tau \in\Bc\mid \Tc=s}\Qbb{\Tc\in ds}=\int_{a}^{b}\Pbb{\tau \in\Bc\mid \Oc_s}\Pbb{\mathfrak{C}\in ds},
\end{align}
where $\Bc\in\Fc_a,\,b>a>0$. Clearly, $\tau_s=\infty, s\geq\Tc$.

Next we consider the original Brownian motion $B$. To show that $B$ converges under $\Pb_t$ to the process specified in the theorem we shall rely on the so-called Ito's representation of the Brownian motion via its excursions away from zero. This is well developed and explained in the proof of Theorem 2 in \cite{BB1} and we shall be brief on some details. Let us denote by $$\mathrm{E}=\lbcurlyrbcurly{\epsilon\in\Cc\lbrb{0,\infty};\epsilon(0)=0;\labsrabs{\epsilon(s)}>0,\,0<s<\zeta\lbrb{\epsilon}\leq\infty;\epsilon(s)=0,\,s\geq \zeta\lbrb{\epsilon}}$$ the space of excursions of the Brownian motion away from zero. $\zeta\lbrb{\epsilon}$ is called the life-time of the excursion $\epsilon\in\mathrm{E}$. Consider the jump process of $\tau$, i.e. \[\Delta_\tau=\lbcurlyrbcurly{\Delta_s}_{s\geq 0}=\twopartdef{\tau_s-\tau_{s-}}{\tau_s-\tau_{s-}>0}{0}{\tau_s-\tau_{s-}=0}\]
with $\tau_t=\sum_{s\leq t}\Delta_s$.
The process $\lbcurlyrbcurly{\lbrb{s,\Delta_s}}_{s\geq 0}$ defines a Poisson point process on $[0,\infty)\times\Rb^+$ which we expand in the following manner. Conditionally on $\Delta_s=x>0$ we sample $\epsilon_s$ from $\mathrm{E}_x=\mathrm{E}\bigcap\lbcurlyrbcurly{\epsilon\in\mathrm{E}:\,\zeta(\epsilon)=x}$ according to the measure of a Brownian meander (namely Brownian bridges conditioned not to cross zero, see \cite{DuIgMi} for more detail) of length $x>0$ which is either positive or negative with equal probability. Otherwise, when $\Delta_s=0,$ we set $\epsilon_s\equiv0$. Thus we have the process defined as follows  \[\mathrm{U}:=\lbcurlyrbcurly{\lbrb{\Delta_s,\epsilon_s}}_{s\geq 0}=\twopartdef{\lbrb{\tau_s-\tau_{s-},\epsilon_s}}{\tau_s-\tau_{s-}>0}{\lbrb{0,0}}{\tau_s-\tau_{s-}=0}.\]
The first passage time process of $\tau$ across all levels $t>0$ coincides with the local time at zero of the original Brownian motion $B$, namely $\lbcurlyrbcurly{L_t}_{t\geq 0}$. Then  $\mathrm{V}=\lbcurlyrbcurly{\lbrb{\tau_s,\epsilon_s}}_{s\geq 0}$ defines a standard Brownian motion via the definition $B'_{u}=\epsilon_{\tau_{L(u-)}}\lbrb{u-\tau_{L(u-)}}, u\geq 0$, i.e. $B'\stackrel{d}{=}B$  and vice versa, decomposing the path of $B$ into excursions away from zero via the Ito's excursion representation we can obtain $\mathrm{V}$. Recall that $\Delta^{g(t)}_1=\inf\lbcurlyrbcurly{s>0:\,\tau_s-\tau_{s-}>g(t)}$ and consider the stopped process $\mathrm{V}^t=\lbcurlyrbcurly{\lbrb{\tau_s,\epsilon_s}}_{s< \Delta^{g(t)}_1}$ and the extended process $\mathfrak{V}^t=\lbrb{\mathrm{V}^t,\Delta^{g(t)}_1,\epsilon_{\Delta^{g(t)}_1}}$. Note that from $\mathfrak{V}^t$ we can construct the Brownian motion until and including the first excursion away from zero of life-time longer than $g(t)$ and vice versa. We shall show that under $\Pb_t$ both $\mathrm{V}$ and $\mathfrak{V}_t$ have the same limit which coincides with the explicit process of the theorem. Since $\mathfrak{V}_t$ takes values in $\Dc\lbrb{0,\infty}\times \mathrm{E}^{\infty}\times\Rb\times\mathrm{E}$ and conditionally on $\lbcurlyrbcurly{\pi_h\lbrb{\tau}=\vartheta,\Delta^{g(t)}_1=s>h}$, where $\pi_h\lbrb{\tau}=\lbcurlyrbcurly{\tau_s}_{s\leq h},\,\vartheta\in\Dc\lbrb{0,h}$, the excursion process forms an independent sampling of Brownian meanders  with given lengths $\lbrb{\vartheta(u)-\vartheta(u-)}_{u\leq h}$ until time $h$, we get that for each fixed triplet $b>a>h>0, \Bc\subseteq\Fc_h$ and bounded continuous functionals $F_1:\mathrm{E}^\infty\mapsto \Rb$ restricted to all excursions up to $h$, i.e. to $\pi_h\lbrb{\epsilon}=\lbcurlyrbcurly{\epsilon_s}_{s\leq h},$ and $F_2:\mathrm{E}\mapsto \Rb$
\begin{align*}
&\Eb^t\lbrb{\ind{\pi_h\lbrb{\tau}\in\Bc}\ind{\Delta^{g(t)}_1\in\lbrb{a,b}}F_1\lbrb{\pi_h\lbrb{\epsilon}}F_2\lbrb{\epsilon_{\Delta^{g(t)}_1}}}=\\
&\int_{\Bc}\int_{a}^{b}\Eb\lbrb{F_1\lbrb{\pi_h\lbrb{\epsilon}}F_2\lbrb{\epsilon_{u}\ind{\zeta(\epsilon_u)>g(t)}}\mid \pi_h(\tau)=\vartheta;\Delta^{g(t)}_1=u}\times\\
&\frac{\Pbb{\pi_h(\tau)\in d\vartheta,\Delta^{g(t)}_1\in du;\,\Oc_t}}{\Pbb{\Oc_t}}
\end{align*}
where we have used additionally that $\ind{\Oc_t}$  is a functional of $\pi_t\lbrb{\tau}$ only. However, we have that conditionally on $\lbcurlyrbcurly{\pi_h(\tau)=\vartheta;\Delta^{g(t)}_1=u}$ the law of $\epsilon_u$ is independent of $\pi_h\lbrb{\epsilon}$ and equals in law the law of a Brownian excurion conditioned on its life-time being longer than $g(t)$, say $\epsilon^{g(t)}_u$, and  $\pi_h(\epsilon)$ equals in law $\pi_h(\epsilon^t)$, where $\epsilon^t$ is an excursion process consisting of excursions whose individual life-time does not exceed $g(t)$. Therefore,
\begin{align*}
&\Eb^t\lbrb{\ind{\pi_h\lbrb{\tau}\in\Bc}\ind{\Delta^{g(t)}_1\in\lbrb{a,b}}F_1\lbrb{\pi_h\lbrb{\epsilon}}F_2\lbrb{\epsilon_{\Delta^{g(t)}_1}}}=\\
&\int_{\Bc}\int_{a}^{b}\Eb_{\vartheta,u}\lbrb{F_1\lbrb{\pi_h\lbrb{\epsilon^t}}}\Eb\lbrb{F_2\lbrb{\epsilon^{g(t)}_{u}}}\frac{\Pbb{\pi_h(\tau)\in d\vartheta,\Delta^{g(t)}_1\in du;\,\Oc_t}}{\Pbb{\Oc_t}}.
\end{align*}
 By $\Eb_{\vartheta,u}$ we understand the expectation under sampling of Brownian meanders  given the position of their arrival (the start of the excursion of the Brownian motion away form zero) and their length. Since the measure of integration converges as $t\to\infty$ we conclude that
\begin{align*}
	&\lim\ttinf{t}\Eb^t\lbrb{\ind{\pi_h\lbrb{\tau}\in\Bc}\ind{\Delta^{g(t)}_1\in\lbrb{a,b}}F\lbrb{\pi_h\lbrb{\epsilon};\epsilon_{\Delta^{g(t)}_1}}}=\\
	&\int_{\Bc}\int_{a}^{b}\Eb_{\vartheta,u}\lbrb{F_1\lbrb{\pi_h\lbrb{\epsilon}}}\Eb\lbrb{F_2\lbrb{\epsilon^{\infty}_{u}}}\Qbb{\pi_h(\tau)\in d\vartheta,\Tc\in du},
\end{align*}
where it is proved in \cite[Proof of Thm.2]{BB1} that $\epsilon^{g(t)}$ converges to, as $t\to\infty$, to a Bessel three process with random sign, denoted here by $\epsilon^\infty$ and hence $\lim\ttinf{t}\Ebb{F_2\lbrb{\epsilon^{g(t)}_{u}}}=\Ebb{F_2\lbrb{\epsilon^\infty}}$. Also clearly since $h>0$ is fixed $$\lim\ttinf{t}\Eb_{\vartheta,u}\lbrb{F_1\lbrb{\pi_h\lbrb{\epsilon^t}}}=\Eb_{\vartheta,u}\lbrb{F_1\lbrb{\pi_h\lbrb{\epsilon}}}.$$ We note that as a consequence one obtains that $\epsilon^\infty$ is independent of $\pi_h\lbrb{\epsilon}$ conditionally on $\lbcurlyrbcurly{\pi_h(\tau)=\vartheta;\Delta^{g(t)}_1=u}$. Given the description of the joint law of $\lbrb{\pi_\Tc\lbrb{\tau},\Tc}$ under $\Qb$, see \eqref{eq:lawJoint} we conclude that the construction $B'_u=\epsilon_{\tau_{L(u-)}}\lbrb{u-\tau_{L(u-)}}, \tau_{\Tc-}\geq u\geq 0$ under $\Qb$ is nothing else but the Brownian motion with its inverse local time running up to the time of the clock $\mathfrak{C}$ conditioned on $\lbcurlyrbcurly{\tau_s>g(s),\,s\leq\mathfrak{C}}$. The process $\epsilon^\infty$ is an independent Bessel three process with a random sign. Splicing $\epsilon^\infty$ at time $\tau_{\mathfrak{C}-}$ gives the process of the theorem. The uniqueness follows from the uniqueness of the law of $\lbrb{\pi_{\Tc}\lbrb{\tau},\Tc}$ and the independence of the limit as $t\to\infty$ of $\Eb_{\vartheta,u}\lbrb{\cdot}$ above. Thus under $\Pb_t$, $\mathfrak{V}_t$ converges to the process defined in the theorem. Let us show that under $\Pb_t$, $\mathrm{V}$ converges to the same process. This follows easily from above since 
\[\lim\ttinf{b}\lim\ttinf{t}\Pbb{\Delta^{g(t)}_1>b\mid\Oc_t}=\lim\ttinf{b}\frac{\int_{b}^{\infty}\Pbb{\Oc_s}du}{\Phi(\infty)}=0\]
and since $\tau_{\Delta^{g(t)}_1}\geq g(t)\uparrow\infty,\,t\to\infty$. This completes the proof of the theorem as the construction of the Brownian motion from $\mathrm{V}$ converges thus to the process of the theorem.
\end{proof}
The proof of Corollary \ref{cor:transience} is immediate from Theorem \ref{thm:transience}.
 \section{Proofs for Section \ref{sec:recurrence}}
We prove Theorem \ref{thm:recurrent} in several steps. First we show that  $\Pbb{\tau\in\cdot\mid \Oc_t}\to \Qbb{\cdot}$ and we describe the law of $\tau$ under $\Qb$.
Recall that $\Phi^h_y(t)=\int_{0}^{t}\Pbb{\tau_{s}>g(s+h)-y,s\leq v}dv$, see the introduced notation around \eqref{eq:relationO}. Then the following claim holds:
\begin{proposition}\label{cor:densityUnderQ}
	Under $\Pb_t$ the inverse local time converges, as $t\to\infty$, to an increasing pure-jump process with law $\Qb$ which we call the inverse local time under $\Qb$.
	Under $\Qb$ the inverse local time has a density given, for $y>g(h)\vee 0$, $A>3$ and $t(A)$ such that $g(t(A))=1+2A^{-1}$, by
	\begin{align}\label{eq:densityInverseLocalTime}
	&\Qb(\tau_h\in dy):=q_h(y)\Pbb{\tau_h\in dy,\Oc_h}=\frac{\Phi^h_y\lbrb{f(Ay)\vee t(A)}}{\Phi(1)}e^{-\int_{1}^{f(Ay)\vee t(A)}\frac{2K}{\gN{s}}ds}\times\\
	\nonumber &e^{\int_{f(Ay)\vee t(A)}^{\infty}\lbrb{\frac{2K}{\sqrt{g(s+h)-y}}-\frac{2K}{\gN{s}}}ds}e^{-\int_{1}^{\infty}\rho(s)ds+\int_{f(Ay)\vee t(A)}^{\infty}\rho^h_y(s)ds}\Pbb{\tau_h\in dy,\Oc_h},
	\end{align}
	where it is part of the proof that all quantities involved are finite and $\rho^h_y(s)$, $\rho(s)$ are defined in \eqref{eq:rho} and \eqref{eq:rho_h_y}. Furthermore, for any $B\subset\Oc_h$ and $B\in\Fc_h$ we have that
	\begin{equation}\label{eq:h-transformRecurrence}
	\Qbb{B}=\Ebb{q_h(\tau_h);B}.
	\end{equation} 
	Finally, the function $q_h:\lbrb{g(h)\vee 0,\infty}\mapsto \lbrb{0,\infty}$ is non-decreasing for every $h>0.$
\end{proposition}
\begin{proof}[Proof of Proposition \ref{cor:densityUnderQ}]
	We write using the Markov property, for any $B\subset \Oc_h;\,B\in\Fc_h,$
	\begin{align*}
	&\Pbb{\tau_h\in dy;B\Big| \Oc_t}=\frac{\Pbb{\Oc_t;\,\tau_h\in dy;\,B}}{\Pbb{\Oc_t}}=\frac{\Pbb{\Oc_{t-h}(h,y)}}{\Pbb{\Oc_t}}\Pbb{\tau_h\in dy;B,\,\Oc_h},
	\end{align*}
	where we recall that $\Oc_{t-h}(h,y)=\{\tau_s>g_{y,h}(s),s\leq t-h\}$ and $g_{y,h}(t)=g(t+h)-y$. Clearly, for every $t>h$, conditionally on $\Oc_t$,  $y>g(h)\vee 0$.
	Fix $y\in\lbrb{g(h)\vee 0,\infty}$, it remains to show the following limit
	\begin{equation}\label{eq:limit}
	\lim\ttinf{t}\frac{\Pbb{\Oc_{t-h}(h,y)}}{\Pbb{\Oc_t}}=q_h(y).
	\end{equation}
	However, using equation \eqref{eq:asymptoticInFiniteCase-1} of Theorem \ref{thm:asymptoticInFiniteCase} we get that
	\[\frac{\Pbb{\Oc_{t-h}(h,y)}}{\Pbb{\Oc_t}}\sim\frac{\Phi^h_y(t-h)\sqrt{g(t)}}{\Phi(t)\sqrt{g(t-h+h)-y}}\sim\frac{\Phi^h_y(t-h)}{\Phi(t)}.\]
	Next, we employ, for $t>\lbrb{t(A)+h}\vee f(Ay)>1$, the modified solutions to \eqref{eq:generalODE}
	\[\Phi^h_y(t-h)=\Phi^h_y\lbrb{f(Ay)\vee t(A)}e^{\int_{f(Ay)\vee t(A)}^{t-h}\frac{2K}{\sqrt{g(s+h)-y}}ds+\int_{f(Ay)\vee t(A)}^{t-h}\rho^h_y(s)ds}\]
	\[\Phi(t)=\Phi(1)e^{\int_{1}^{f(Ay)\vee t(A)}\frac{2K}{\gN{s}}ds+\int_{f(Ay)\vee t(A)}^{t}\frac{2K}{\sqrt{g(s)}}ds+\int_{1}^{t}\rho(s)ds}.\]
	Clearly then, on $y>g(h)$,
	\begin{align*}&q_h(y)=\lim\ttinf{t}\frac{\Phi^h_y(t-h)}{\Phi(t)}=\\
	&\frac{\Phi^h_y\lbrb{f(Ay)\vee t(A)}}{\Phi(1)}e^{-\int_{1}^{f(Ay)\vee t(A)}\frac{2K}{\gN{s}}ds}e^{\int_{f(Ay)\vee t(A)}^{\infty}\lbrb{\frac{2K}{\sqrt{g(s+h)-y}}-\frac{2K}{\gN{s}}}ds}e^{-\int_{1}^{\infty}\rho(s)ds+\int_{f(Ay)\vee t(A)}^{\infty}\rho^h_y(s)ds},
	\end{align*}
	which proves the existence of density $q_h(y)$ with respect to $\Pbb{\tau_h\in dy,\Oc_h}$. The finiteness of $\int_{f(Ay)\vee t(A)}^{\infty}\lbrb{\frac{2K}{\sqrt{g(s+h)-y}}-\frac{2K}{\gN{s}}}ds$ follows from the axillary Lemma \ref{lem:finitenessIntegral} below whereas the finiteness of $\int_{f(Ay)\vee t(A)}^{\infty}\rho^h_y(s)ds$ follows from the bound \eqref{eq:uniformHestimate} of Lemma \ref{lem:uniformIntegrability} which holds under the assumptions for $g$ under Proposition \ref{cor:densityUnderQ}. Finally, the fact that $q_h(y)$ is non-decreasing in $y>g(h)\vee 0$ follows from the fact that for $y_2>y_1>g(h)\vee 0$ we have that for any $t>h$, $\Oc_t\lbrb{h,y_1}\subseteq\Oc_t\lbrb{h,y_2}$ since $g_{y_1,h}(t)=g(t+h)-y_1\geq g_{y_2,h}(t)=g(t+h)-y_2$.
\end{proof}
\begin{lemma}\label{lem:finitenessIntegral}
	Let $f$ satisfy condition \eqref{eq:growthCondition}, i.e. $f(x)/\sqrt{x}\downarrow 0.$ Then we have that, for any $h>0,\,y>g(h)\vee 0,\,A>3$
	\begin{align}\label{eq:finitenessIntegral}
		\nonumber\int_{f(Ay)\vee t(A)}^{\infty}\lbrb{\frac{2K}{\sqrt{g(s+h)-y}}-\frac{2K}{\gN{s}}}ds&\leq\int_{f(Ay)\vee t(A)}^{\infty}\lbrb{\frac{1}{\sqrt{g(s)-y}}-\frac{1}{\gN{s}}}ds\\
		&<\frac{f(y)}{2\sqrt{y(1-\frac{1}{A})}}.
	\end{align}
\end{lemma}
\begin{proof}[Proof of Lemma \ref{lem:finitenessIntegral}]
	Fix $h>0,y>g(h)\vee 0$. Then since for $g(s)>g(f(Ay))>Ay>3y$ and the general inequality holds $1-\sqrt{1-x}\leq x,x\in(0,1)$, we have that
\begin{align*}
&\int_{f(Ay)\vee t(A)}^{\infty}\lbrb{\frac{1}{\sqrt{g(s+h)-y}}-\frac{1}{\gN{s}}}ds\leq \int_{f(Ay)\vee t(A)}^{\infty}\lbrb{\frac{1}{\sqrt{g(s)-y}}-\frac{1}{\gN{s}}}ds=\\
&\int_{f(Ay)\vee t(A)}^{\infty}\frac{1-\sqrt{1-\frac{y}{g(s)}}}{\sqrt{g(s)-y}}ds\leq y\int_{f(Ay)\vee t(A)}^{\infty}\frac{1}{g^{3/2}(s)\sqrt{1-\frac{y}{g(s)}}}ds\leq \\
&\frac{y}{\sqrt{1-A^{-1}}}\int_{f(y)}^{\infty}\frac{ds}{g^{3/2}(s)}=\frac{y}{\sqrt{1-A^{-1}}}\int_{y}^{\infty}\frac{f'(s)}{s^{3/2}}ds=\frac{y}{\sqrt{1-A^{-1}}} \frac{f(s)}{s^{3/2}}\Big|^\infty_y+\\
&\frac{3y}{2\sqrt{1-A^{-1}}} \int_{y}^{\infty}\frac{f(s)}{s^{5/2}}ds\leq -\frac{f(y)}{\sqrt{y(1-A^{-1})}}+\frac{3f(y)}{2\sqrt{y(1-A^{-1})}},
\end{align*}
where for the last line we have also used that $f(s)/\sqrt{s}$ is decreasing, i.e. condition \eqref{eq:growthCondition}.
\end{proof}

The proof  of Theorem \ref{thm:recurrent} follows several steps. Fix $h>0$ and we will first prove the following result.
\begin{lemma}\label{lem:ConvergenceOfInverseLocalTime}
Let $f$ satisfy the usual conditions and \eqref{eq:conditionGrowth} holds, i.e. $\liminf\ttinf{t}g(t)/t^2\ln^{8/5+\epsilon}(t)=\infty$. Then for any fixed $h>0$ and all $t>t(A)$ we have that
\begin{equation}\label{eq:qDCT}
	\frac{\Pbb{\Oc_{t-h}(h,y)}}{\Pbb{\Oc_t}}\Pbb{\tau_h\in dy,\,\Oc_h}\leq r_h(y)\text{, $\forall y>g(h)\vee 0$}
\end{equation}
and $\int_{y>g(h)\vee 0}r_h(y)\Pbb{\tau_h\in dy,\,\Oc_h}<\infty$. Consequently, $\forall h>0,$
\begin{equation}\label{eq:Recurrance}
\Qbb{\tau_h\in(g(h),\infty)}=1.
\end{equation}
\end{lemma}
\begin{proof}[Proof of Lemma \ref{lem:ConvergenceOfInverseLocalTime}]
	Since according to Proposition \ref{cor:densityUnderQ} we have that $\frac{\Pbb{\Oc_{t-h}(h,y)}}{\Pbb{\Oc_t}}:\lbrb{g(h)\vee 0,\infty}\mapsto \lbrb{0,\infty}$ is increasing in $y$ we get that it suffices to consider only $y>g(h+1)>g(1)>1$.
	
	We start with the proof of \eqref{eq:qDCT}.  Choose $A>3$ and introduce the region $I_1=\{t\leq f(Ay+h)\}=\{y\geq \frac{g(t)}A-h\}$. Next, introduce the functions
	\[w_b(t)=e^{-2b\int_{1}^{t}\frac{1}{\gN{s}}ds}, \text{ $0<b<1$.}\]
Note that from \eqref{eq:solutionODE} when $\liminf\ttinf{t}g(t)/t^2\ln^{8/5+\epsilon}(t)=\infty$ holds we have from \eqref{eq:uniformHestimate} of Lemma \ref{lem:uniformIntegrability} that $\int_{ t(A)}^{\infty}\rho(s)ds<\infty$
\begin{equation}\label{eq:choicew_b}
\sqrt{w_b(t)}\Phi(t)\asymp\Phi(t(A))e^{-b\int_{1}^{t}\frac{1}{\gN{s}}ds+\int_{t(A)}^{t}\frac{2K}{\sqrt{g(s)}}ds+\int_{t(A)}^{t}\rho(s)ds}\asymp e^{(1-b)\int_{1}^{t}\frac{2K}{\gN{s}}ds}\stackrel{\ttinf{t}}\to\infty
\end{equation}
since \eqref{eq:finitenessPhi} does not hold, i.e. 	$\int_{1}^{\infty}\frac{2K}{\gN{s}}ds=\infty$. We define the region $I_b=\{y>g(t)w_b(t)-h\}$ and since $w_b(t)\geq \frac1A$ once $t\geq t(b,A)$ is big enough we get for all those $t\geq t(b,A)$ that $I_1\subseteq I_b$. Then the trivial estimate
\begin{align}\label{eq:estimateRatioBigRegion}
\frac{\Pbb{\Oc_{t-h}(h,y)}}{\Pbb{\Oc_t}}\leq \frac{1}{\Pbb{\Oc_t}}\sim \frac{2K\sqrt{g(t)}}{\Phi(t)},
\end{align}
where we have used Theorem \ref{thm:asymptoticInFiniteCase} and \eqref{eq:asymptoticInFiniteCase-1}. However this from \eqref{eq:estimateRatioBigRegion} leads to 
\begin{align}\label{eq:DCTBigRegion}
&\int_{I_b}\frac{\Pbb{\Oc_{t-h}(h,y)}}{\Pbb{\Oc_t}}\Pbb{\tau_h\in dy;A}\lesssim \frac{\gN{t}}{\Phi(t)}\Pbb{\tau_h\geq g(t)w_b(t)-h}=\\
\nonumber& \frac{\gN{t}}{\Phi(t)}\Pbb{\tau_1\geq \frac{g(t)w_b(t)}{h^2}-\frac1h}\leq C(A,h,b)\frac{1}{\Phi(t)\sqrt{w_b(t)}}\leq C(A,h,b):=r^1_h(y)
\end{align}
once $t$ is big enough since \eqref{eq:choicew_b} holds. We conclude that to use DCT we can restrict our attention to the region
$I^c_b=\Rb^+\setminus I_b\subset\{t\geq f(Ay+h)\}=\{y\leq \frac{g(t)}A-h\}$ once $t$ is big enough. Consider  $I^c_b\cap\{y>g(h+1)\}$. Then from  \eqref{eq:uniformHestimate} we get that the solution of \eqref{eq:generalODE} for $t\geq f(Ay)\vee t(A)$ is bounded in the following way uniformly in $y$ and $t$
\begin{align*}
&\Phi^h_y(t)=\Phi^h_y(f(Ay+h))e^{\int_{f(Ay+h)}^{t}\frac{2K}{\sqrt{\lbrb{g(s+h)-y}\vee 1}}ds+\int_{f(Ay+h)}^{t}\rho^h_y(s)ds}\leq\\
& C_1(A,h)\Phi^h_y(f(Ay+h))e^{\int_{f(Ay+h)}^{t}\frac{2K}{\sqrt{\lbrb{g(s+h)-y}\vee 1}}ds}
\end{align*}
since 
\[\int_{f(Ay+h)}^{t}\rho^h_y(s)ds\leq \frac{C(A,h)}{f(y)-h}\int_{e}^{\infty}\frac{ds}{s\ln^{1+\epsilon}(s)}<\frac{C(A,h)}{f(g(h+1))-h}\leq C(A,h).\]
Furthermore, since the elementary $\Phi^h_y(f(Ay+h))\leq f(Ay+h)$ holds, we get the following inequality
\begin{equation}\label{eq:boundPhi}
\Phi^h_y(t)\leq C(A,h)f(Ay+h)e^{\int_{f(Ay+h)}^{t}\frac{2K}{\sqrt{\lbrb{g(s+h)-y}}}ds}.
\end{equation}	
We then use again \eqref{eq:solutionODE} i.e. $\Pbb{\Oc_t}\sim 2K\Phi(t)/\gN{t}$; \eqref{eq:generalODE} to express 
\[\Pbb{\Oc_{t-h}(h,y)}=\frac{2K\Phi^h_y(t-h)}{\sqrt{\lbrb{g(t)-1}\vee 1}}+\Phi^h_y(t-h)\rho^h_y(t-h)=\Phi^h_y(t-h)\lbrb{\frac{2K}{\sqrt{\lbrb{g(t)-y}}}+\rho^h_y(t-h)}\]
once $t>t(A)$ since then $g(t)-y>g(f(Ay+h))=(A-1)y+h>(A-1)g(h+1)+h>A-1>1$ on $I^c_b\cap\lbcurlyrbcurly{y>g(h+1)}$. Simple substitution then yields for $t$ big enough snd $y\in I^c_b\cap\lbcurlyrbcurly{y>g(h+1)}$ since $1-y/g(t)>1-1/A$
\begin{align*}
\frac{\Pbb{\Oc_{t-h}(h,y)}}{\Pbb{\Oc_t}}\leq& C(A,h)\frac{\Phi^h_y(t-h)}{\Phi(t)}\lbrb{\frac{\gN{t}}{\sqrt{g(t)-y}}+\frac{\sqrt{g(t)}\rho^h_y(t-h)}{2K}}=\\
&C(A,h)\lbrb{\frac{\Phi^h_y(t-h)}{\Phi(t)\sqrt{1-y/g(t)}}+\sqrt{g(t)-y}\frac{H^h_y(t-h)}{2K\Phi(t)}}\leq\\
&C(A,h)\frac{\Phi^h_y(t-h)}{\Phi(t)}+C(A,h)\sqrt{g(t)-y}\frac{H^h_y(t-h)}{2K\Phi(t)}.
\end{align*}
Next, since \eqref{eq:uniformComparisonToG} holds with $g_{y,h}$ too and $y\in I^c_b\cap\lbcurlyrbcurly{y>g(h+1)}$ hold we estimate that 
\[\frac{\sqrt{g(t)-y}H^h_y(t-h)}{\Phi(t)}\leq C(A,h)u(t-h)\Phi^h_y(t-h)=o(1)\frac{\Phi^h_y(t-h)}{\Phi(t)}.\]
Therefore, we need to consider only the ratio $\frac{\Phi^h_y(t-h)}{\Phi(t)}$ above. We use \eqref{eq:boundPhi} to estimate the numerator and \eqref{eq:solutionODE} with $t_0=1$ and \eqref{eq:uniformHestimate} of Lemma \ref{lem:uniformIntegrability} for the case $h=y=0$ to get that
\[\Phi(t)=\Phi(1)e^{\int_{1}^{t}\frac{2K}{\sqrt{g(s)}}ds+\int_{1}^{t}\rho(s)ds}\leq Ce^{\int_{1}^{t}\frac{2K}{\sqrt{g(s)}}ds}.\]
Thus, on $I^c_b\cap\lbcurlyrbcurly{y>g(h+1)}$
\begin{align*}
\frac{\Pbb{\Oc_{t-h}(h,y)}}{\Pbb{\Oc_t}}&\leq C(A,h)f(Ay+h)e^{\int_{f(Ay+h)}^{t-h}\frac{2K}{\sqrt{\lbrb{g(s+h)-y}}}ds-\int_{1}^{t}\frac{2K}{\gN{s}}ds}\\
&\leq f(Ay+h)e^{-\int_{1}^{f(Ay+h)+h}\frac{2K}{\gN{s}}ds}e^{\int_{f(Ay+h)+h}^{t}\lbrb{\frac{2K}{\sqrt{\lbrb{g(s)-y}}}-\frac{2K}{\gN{s}}}ds}\\
&\leq C\lbrb{A,h}f(Ay+h)e^{-\int_{1}^{f(Ay+h)+h}\frac{2K}{\gN{s}}ds},
\end{align*}
where for the last exponent we have used \eqref{eq:finitenessIntegral} of Lemma \ref{eq:finitenessIntegral}.
We estimate
\begin{align*}
&f(Ay+h)e^{-\int_{1}^{f(Ay+h)+h}\frac{2K}{\gN{s}}ds}\leq f(Ay+h)e^{-\int_{1}^{f(Ay+h)}\frac{2K}{\gN{s}}ds}=f(Ay+h)e^{-\int_{g(1)}^{Ay+h}\frac{2Kf'(s)}{\sqrt{s}}ds}=\\
&f(Ay+h)e^{-K\int_{g(1)}^{Ay+h}\frac{f(s)}{s^{3/2}}ds- \frac{2Kf(s)}{\sqrt{s}} \Big|_{g(1)}^{Ay+h}}\leq  Cf(Ay+h)e^{-K\int_{g(1)}^{Ay+h}\frac{f(s)}{s^{3/2}}ds}.
\end{align*}
to get on $I^c_b\cap\lbcurlyrbcurly{y>g(h+1)}$
\begin{equation}\label{eq:rh}
\frac{\Pbb{\Oc_{t-h}(h,y)}}{\Pbb{\Oc_t}}\leq Cf(Ay+h)e^{-K\int_{g(1)}^{Ay+h}\frac{f(s)}{s^{3/2}}ds}:=r^2_h(y).
\end{equation}
We set generously $r_h=\lbrb{r^1_h+r^2_h}\ind{y>g(h+1)}$, see\eqref{eq:DCTBigRegion} and \eqref{eq:rh}. However, since $\frac{\Pbb{\Oc_{t-h}(h,y)}}{\Pbb{\Oc_t}}$ is non-decreasing in $y$ and the definition of $r_h$ does not depend on $t$, we could extend without loss of generality to $r_h=\lbrb{r^1_h+r^2_h}\ind{y>g(h)\vee 0}$ as an upper bound and thus get \eqref{eq:qDCT}.
 We have trivially that
\[1=\Pbb{\tau_h\in (g(h),\infty)\Big| \Oc_t}=\int_{g(h)}^{\infty}\frac{\Pbb{\Oc_{t-h}(h,y)}}{\Pbb{\Oc_t}}\Pbb{\tau_h\in dy,\,\Oc_h}.\]

Clearly then using that $\Pbb{\tau_h\in dy,\Oc_h}\leq \Pbb{\tau_h\in dy}=\Pbb{h^2\tau_1\in dy}\leq Cdy/y^{3/2}$, see \eqref{eq:lawTau}, to use the DCT, i.e. show that $\int_{g(h)}^{\infty}r_h(y)\Pbb{\tau_h\in dy,\,\Oc_h}<\infty$ and hence \eqref{eq:Recurrance} we only need to check that
\begin{align*}
& \int_{1}^{\infty}f(Ay+h)e^{-K\int_{g(1)}^{Ay+h}\frac{f(s)}{s^{3/2}}ds}\frac{dy}{y^{3/2}}\leq C(A) \int_{1}^{\infty}f(Ay+h)e^{-K\int_{g(1)}^{Ay+h}\frac{f(s)}{s^{3/2}}ds}\frac{d(Ay+h)}{(Ay+h)^{3/2}}\leq \\
&C(A)\int_{1}^{\infty}f(u)e^{-\int_{g(1)}^{u}\frac{Kf(s)}{s^{3/2}}ds}\frac{du}{u^{3/2}}<\infty
\end{align*}
However, with $\alpha(u)=K f(u)/u^{3/2}$ the integral above can be represented up to a multiplicative constant as
\[\int_{1}^{\infty}\alpha(u)e^{-\int_{g(1)}^{u}\alpha(s)ds}du=-e^{-\int_{g(1)}^{u}\alpha(s)ds}\Big|^{\infty}_{1}<\infty\]
we conclude  \eqref{eq:Recurrance}.

\end{proof}

\begin{proof}[Proof of Theorem \ref{thm:recurrent} and Corollary \ref{cor:densityUnderQ}]
The fact that \eqref{eq:Recurrance} holds implies the statement that any possible weak limit is recurrent as there is no loss of mass at infinity. Similarly \eqref{eq:densityInverseLocalTime} is a consequence of the proof of the limit $q_h(y)$ in Lemma \ref{lem:ConvergenceOfInverseLocalTime}. Moreover, we see that the limit for inverse local time always has the same law. Given that conditional on the size of the jump we fill a Brownian excursion conditioned to have the same length we see that we can in fact pathwise construct the same process so the the limit exists and it is unique. These considerations with the splicing of excursions are even simpler here compared to the transient case as we do not have appearance of explosion time and an infinite excursion. The recurrence follows since the inverse local time does not explode, i.e. under $\Qb$ the process returns to zero with probability one after each time $T>0$. This concludes the proof.
\end{proof}
\begin{proof}[Proof of Theorem \ref{thm:repulsion}]
We observe that $w\in R_g\iff$
\begin{equation}\label{eq:Rg}
\lim\ttinf{h}\Qb\lbrb{\tau_h\in\lbrb{g(h),w(h)g(h)}}=\lim\ttinf{h}\int_{g(h)}^{g(h)w(h)}q_h(y)\Pbb{\tau_h\in dy;\,\Oc_h}=0.
\end{equation}
Given the expression for $q_h(y)$, see  \eqref{eq:densityInverseLocalTime}, we note that thanks to Lemma \ref{lem:finitenessIntegral} we have that 
\[\int_{f(Ay)\vee t(A)}^{\infty}\lbrb{\frac{1}{\sqrt{g(s+h)-y}}-\frac{1}{\gN{s}}}ds<\infty\]
and thanks to Lemma \ref{lem:uniformIntegrability} with $h=y=0$, $\int_{1}^{\infty}\rho(s)ds<\infty$. Thus, choosing $h$ big enough that $f(Ay)>f(Ag(h))>t(A)$
\begin{equation}\label{eq:qAsymp}
q_h(y)\asymp\Phi^h_y\lbrb{f(Ay)}e^{-\int_{1}^{f(Ay)}\frac{2K}{\gN{s}}ds}e^{\int_{f(Ay)}^{\infty}\rho^h_y(s)ds},
\end{equation}
which thanks to \eqref{eq:uniformHestimate} of Lemma \ref{lem:uniformIntegrability} once $y>g(h+1)$ is augmented to 
\begin{equation}\label{eq:qh}
q_h(y)\asymp\Phi^h_y\lbrb{f(Ay)}e^{-\int_{1}^{f(Ay)}\frac{2K}{\gN{s}}ds}
\end{equation}
since $\int_{1}^{\infty}\rho^h_y(s)ds\leq C(A)/(f(g(h+1))-h)=C(A)$. 
To prove \eqref{eq:Rg} we first consider
\begin{equation}\label{eq:Rg1}
\int_{g(h)}^{g(h+1)}q_h(y)\Pbb{\tau_h\in\,dy;\,\Oc_h}\leq q_h\lbrb{h+1}\Pbb{\Oc_h}\stackrel{h\to\infty}{\sim} q_h\lbrb{h+1}\frac{2K\Phi(h)}{\gN{h}}.
\end{equation}
However, from \eqref{eq:qh}, the trivial $\Phi\lbrb{f(Ag(h+1))}\leq f(Ag(h+1))\leq \sqrt{A}(h+1)$ since \eqref{eq:growthCondition} holds and $f(Ag(h+1))>h$, we get that
$$q_h(h+1)\lesssim h e^{-\int_{1}^{h}\frac{2K}{\gN{s}}ds}.$$
 The relation $\Phi(h)\asymp e^{\int_{1}^{h}\frac{2K}{\gN{s}}}ds$ coming from \eqref{eq:asymptoticInFiniteCase-2} easily implies that  $$q_h\lbrb{h+1}\frac{2K\Phi(h)}{\gN{h}}\lesssim \frac{h}{\sqrt{g(h)}}=o(1).$$ Thus the portion \eqref{eq:Rg1} never contributes to \eqref{eq:Rg}. We are then free to use the asymptotic relation \eqref{eq:qAsymp} for the interval $y\in\lbrb{g(h),w(h)g(h)}:=I$ which we split into $I_1=\lbrb{g(h),20g(h)}$ and $I_2=I\setminus I_1$. We then get that \eqref{eq:Rg} can be checked on $I_1,I_2$.
Let us start with $I_1$. We get using $\Phi^h_y\lbrb{f(Ay)}\leq f\lbrb{Ay}\leq \sqrt{A}f(y)$, since \eqref{eq:growthCondition} holds, that 
\begin{align*}
\int_{I_1}q_h(y)\Pbb{\tau_h\in dy;\,\Oc_h}&\lesssim \int_{I_1}f(y)e^{-\int_{1}^{f(Ay)}\frac{2K}{\gN{s}}ds}\Pbb{\tau_h\in dy;\,\Oc_h}\\
&\leq f(20g(h))e^{-\int_{1}^{f(g(h))}\frac{2K}{\gN{s}}ds}\Pbb{\Oc_h}\lesssim \frac{h\Phi(h)}{\sqrt{g(h)}}e^{-\int_{1}^{f(g(h))}\frac{2K}{\gN{s}}ds}\\
&\asymp \frac{h}{\gN{h}}=o(1),
\end{align*} 
where for the last relations we have used the asymptotic relations \eqref{eq:asymptoticInFiniteCase-1}, \eqref{eq:asymptoticInFiniteCase-2}. Therefore, the portion on $I_1$ is always negligible and we obtain thanks to Lemma \ref{lem:strongRepulsion} and \eqref{eq:qh} that  
\begin{align}\label{eq:Rg2}
&w\in R_g\iff\\
\nonumber&\lim\ttinf{h}\Qb\lbrb{\tau_h\in\lbrb{20g(h),w(h)g(h)}}=\lim\ttinf{h}\int_{20}^{w(h)}q_h(yg(h))\Pbb{\tau_h\in g(h)dy;\,\Oc_h}\\
\nonumber &\lim\ttinf{h}\Pbb{\Oc_h}\int_{20}^{w(h)}q_h(g(h)y)\frac{dy}{y^{3/2}}=\lim\ttinf{h}\Phi(h)\int_{20g(h)}^{w(h)g(h)}\Phi^h_y\lbrb{f(Ay)}e^{-\int_{1}^{f(Ay)}\frac{2K}{\gN{s}}ds}\frac{dy}{y^{3/2}}=0.
\end{align}  
 Let us first show that the convergences in \eqref{eq:repulsionRegion} is sufficient for $w\in R_g$. So far it was supposed that $A>3$. Choose $A\in\lbrb{3,20}$. Then using the inequality $\Phi^h_y\lbrb{f(Ay)}\leq f(Ay)$ and the change of variables $Ay\mapsto u$ we get easily that
 \begin{align}\label{eq:sufficentRep}
 	\nonumber&\lim\ttinf{h}\Phi(h)\int_{20g(h)}^{w(h)g(h)}\Phi^h_y\lbrb{f(Ay)}e^{-\int_{1}^{f(Ay)}\frac{2K}{\gN{s}}ds}\frac{dy}{y^{3/2}}\leq\\
 	&\sqrt{A}\lim\ttinf{h}\Phi(h)\int_{g(h)}^{w(h)g(h)}f(u)e^{-\int_{1}^{f(u)}\frac{2K}{\gN{s}}ds}\frac{du}{u^{3/2}}.
 \end{align}
An integration by parts gives us
\[\int_{1}^{f(y)}\frac{1}{\gN{s}}ds=\frac{f(y)}{\sqrt{y}}-1+\frac{1}{2}\int_{1}^{y}\frac{f(s)}{s^{3/2}}ds\sim \frac{1}{2}\int_{1}^{y}\frac{f(s)}{s^{3/2}}ds \]
since $f(s)/\sqrt{s}\downarrow 0$, see \eqref{eq:growthCondition} and $I(f)=\int_{1}^{\infty}\frac{f(s)}{s^{3/2}}ds=\infty$.
 Then, clearly,
\begin{align*}
\int_{y=g(h)}^{w(h)g(h)}f(u)e^{-\int_{1}^{f(u)}\frac{2K}{\gN{s}}ds}\frac{du}{u^{3/2}}&\asymp \int_{y=g(h)}^{w(h)g(h)}f(y)e^{-\frac{1}{2}\int_{1}^{y}\frac{2Kf(s)}{s^{3/2}}ds}\frac{dy}{y^{3/2}}
\\&=-\frac{1}{K}e^{-\frac{1}{2}\int_{1}^{y}\frac{2Kf(s)}{s^{3/2}}ds}\Bigg|_{g(h)}^{g(h)w(h)}.
\end{align*}
Expressing back $\frac{1}{2}\int_{1}^{y}\frac{f(s)}{s^{3/2}}ds$ in terms $\int_{1}^{f(y)}\frac{1}{\sqrt{g(s)}}ds$ and noting that $\Phi(h)\asymp e^{\int_{1}^{h}\frac{2K}{\gN{s}}ds}$ we get that the right-hand side of \eqref{eq:sufficentRep} behaves asymptotically as 
\begin{align*}
&\Phi(h)\int_{y=g(h)}^{w(h)g(h)}f(y)e^{-\int_{1}^{f(y)}\frac{2K}{\gN{s}}ds}\frac{dy}{y^{3/2}}\asymp \lbrb{1-e^{-\int_{g(h)}^{g(h)w(h)}\frac{K}{\gN{s}}ds}},
\end{align*}
which proves the sufficiency of \eqref{eq:repulsionRegion} for $w\in R_g$.

The necessity part of \eqref{eq:repulsionRegion} is trickier. Assume that $w\in R_g$ and hence the right-hand side of \eqref{eq:Rg2} holds. We recall the inequality \eqref{eq:lowerPhi}, namely $\Phi^y_h\lbrb{f(Ay)}\geq f(y)-h, y>g(h)\vee 0$. Let us feed this inequality in the right-hand side of \eqref{eq:Rg2} and consider first only the term introduced by $-h$. Recall that $\Phi(h)\asymp e^{\int_{1}^{h}\frac{2K}{\sqrt{g(s)}}ds}$. Then
\begin{align*}
&\lim\ttinf{h}h\Phi(h)\int_{20g(h)}^{w(h)g(h)}e^{-\int_{1}^{f(Ay)}\frac{2K}{\gN{s}}ds}\frac{dy}{y^{3/2}}\lesssim \\
&\lim\ttinf{h}h\Phi(h)\int_{y=20g(h)}^{w(h)g(h)}e^{-\int_{1}^{f(u)}\frac{2K}{\gN{s}}ds}\frac{du}{u^{3/2}}\leq\\ &\lim\ttinf{h}h\Phi(h)e^{-\int_{1}^{h}\frac{2K}{\gN{s}}ds}\int_{y=20g(h)}^{\infty}\frac{du}{u^{3/2}}\lesssim \lim\ttinf{h}\frac{h}{\gN{h}}=0.
\end{align*}
Therefore this term never contributes to \eqref{eq:Rg2}. Then, we need only discuss the term $f(y)$ introduced by $\Phi^y_h\lbrb{f(Ay)}\geq f(y)-h, y>g(h)\vee 0$ knowing that
\begin{align*}
&0=\lim\ttinf{h}\Phi(h)\int_{20g(h)}^{w(h)g(h)}f(y)e^{-\int_{1}^{f(Ay)}\frac{2K}{\gN{s}}ds}\frac{dy}{y^{3/2}}\\
&= \lim\ttinf{h}\sqrt{A}\Phi(h)\int_{y=\frac{20g(h)}A}^{\frac{w(h)g(h)}A}f\lbrb{\frac{u}A}e^{-\int_{1}^{f(u)}\frac{2K}{\gN{s}}ds}\frac{du}{u^{3/2}}\\
&\geq \lim\ttinf{h}\Phi(h)\int_{y=\frac{20g(h)}A}^{\frac{w(h)g(h)}A}f(u)e^{-\int_{1}^{f(u)}\frac{2K}{\gN{s}}ds}\frac{du}{u^{3/2}}=\\
&\lim\ttinf{h}-\frac{\Phi(h)}{K}e^{-\frac{1}{2}\int_{1}^{y}\frac{2Kf(s)}{s^{3/2}}ds}\Bigg|_{g(h)}^{g(h)w(h)}-\\
&\lim\ttinf{h}\Phi(h)\lbrb{\int_{g(h)}^{\frac{20g(h)}A}f(u)e^{-\int_{1}^{f(u)}\frac{2K}{\gN{s}}ds}\frac{du}{u^{3/2}}+\int_{\frac{w(h)g(h)}A}^{w(h)g(h)}f(u)e^{-\int_{1}^{f(u)}\frac{2K}{\gN{s}}ds}\frac{du}{u^{3/2}}},
\end{align*}
where we have used the fact that $f(s)/s^{1/2}\downarrow 0$ and thus $f(u/A)\geq A^{-1/2}f(u)$. We note that the very first expression in the last quantity is precisely the expression discussed in the case of sufficiency and it will converge to zero if the last two terms converge to zero and thus the necessity of \eqref{eq:repulsionRegion} will follow. We trivially estimate
\begin{align*}
	&\lim\ttinf{h}\Phi(h)\int_{g(h)}^{\frac{20g(h)}A}f(u)e^{-\int_{1}^{f(u)}\frac{2K}{\gN{s}}ds}\frac{du}{u^{3/2}}\lesssim\\
	&\lim\ttinf{h} f\lbrb{\frac{20g(h)}{A}}\int_{g(h)}^{\frac{20g(h)}A}\frac{du}{u^{3/2}}\lesssim\lim\ttinf{h}\frac{h}{\gN{h}}=0.
\end{align*}
The third term is also computed using identical calculations as above as
\begin{align*}
	&\Phi(h)\int_{\frac{w(h)g(h)}A}^{w(h)g(h)}f(u)e^{-\int_{1}^{f(u)}\frac{2K}{\gN{s}}ds}\frac{du}{u^{3/2}}=-\frac{\Phi(h)}{K}e^{-\frac{1}{2}\int_{1}^{y}\frac{2Kf(s)}{s^{3/2}}ds}\Bigg|_{\frac{g(h)w(h)}A}^{g(h)w(h)}\asymp\\
	& -e^{\int_{1}^{h}\frac{2K}{\gN{s}}ds}e^{-\int_{1}^{f(y)}\frac{2K}{\gN{s}}ds}\Bigg|_{\frac{g(h)w(h)}A}^{g(h)w(h)}=-e^{-\int_{h}^{f(y)}\frac{2K}{\gN{s}}ds}\Bigg|_{\frac{g(h)w(h)}A}^{g(h)w(h)}=o(1).
\end{align*}
The last claims follows from $\int_{1}^{\infty}\frac{1}{\gN{s}}ds=\infty.$ This finally concludes the proof of Theorem \ref{thm:repulsion}.
\end{proof}

The strong repulsion depends on the following Lemma which studies the measures $\Pbb{\tau_h\in g(h)dy ,\Oc_h}$.
\begin{lemma}\label{lem:strongRepulsion}
Let $\sigma_h(dy)=o_h(y)dy=\Pbb{\tau_h\in g(h)dy ,\Oc_h}$ be a measure on $(1,\infty)$. Then for any $h\geq h_0$ big enough there are constants $0<c<1<C<\infty$ such that for $y>20$ we have that
\begin{equation}\label{eq:strongRepulsion}
\frac{c}{y^{3/2}}\Pbb{\Oc_h}\leq o_h(y)\leq \frac{C}{y^{3/2}}\Pbb{\Oc_h}.
\end{equation}
\end{lemma}
\begin{proof}[Proof of Lemma \ref{lem:strongRepulsion}]
The absolute continuity of $\sigma_h(dy)$ follows immediately from $\sigma_h(dy)\ll \Pbb{\tau_h\in dy}\ll dy$. For the proof we introduce the quantities $T:=T_h=\inf\{t>0:\,\tau_t>g(h)\}$, in the usual sense $\Delta=\Delta_1=\inf\{t>0:\,\tau_t-\tau_{t-}>g(h)\}$ and $S_{\Delta}=\tau_{\Delta}-\tau_{\Delta-}$. We then have that 
\begin{equation}\label{eq:sizeOfJump}
\Pbb{S_{\Delta}\in g(h)dy}=\frac{\Pi(g(h)dy)}{\PP(g(h))}=\frac{dy}{y^{3/2}},
\end{equation}
which is a standard property for any \LL process, namely conditionally that a jump exceeds a level $a>0$ than its size is independent of the time of the jump and the past of the process and its distribution is given by the first ratio in \eqref{eq:sizeOfJump}. The second ratio in \eqref{eq:sizeOfJump} holds only in this special instance of a stable subordinator of index $1/2$.
Furthermore, denote by $\tau_{T-}$ the position prior to the passage time. Finally, note that we have the immediate identity from \eqref{eq:lawTau}
\begin{equation}\label{eq:localTau}
\Pbb{\tau_v\in a du}= \frac{1}{u\sqrt{2\pi av^{-2} u}}e^{-\frac{1}{av^{-2}u}}\,\,du.
\end{equation}
We consider the measure $\sigma_h(dy)$ on three possibly overlapping regions. We start with $\sigma^{1}_h(dy):=\sigma_h\lbrb{dy,\tau_T\leq 2g(h)}$. Disintegrating on $\tau_{T}\in\lbrb{g(h),2g(h)}$ in this scenario we get, using $\tau_{h}=\tau_T+\tau_{h}-\tau_T\stackrel{d}{=}\tau_T+\tau'_{h-T}$ on $\lbcurlyrbcurly{T\leq h}$ with $\tau'$ an independent copy of $\tau$, that
\begin{align}\label{eq:estimateRepuslion1}
\nonumber&\sigma^{1}_h(dy)=o^{1}_h(y)dy=\int_{s=0}^{h}\int_{u=1}^{2}\Pbb{\tau_{h-s}\in g(h)\lbrb{dy-u}}\Pbb{T\in ds,\,\tau_{T}\in g(h)du;\,\Oc_h}\leq \\
\nonumber&C\int_{s=0}^{h}\int_{u=1}^{2}\frac{h-s}{\sqrt{g(h)}(y-u)^{3/2}}\Pbb{T\in ds,\,\tau_{T}\in g(h)du;\,\Oc_h}dy\leq\\
& C_1\frac{h}{\gN{h}}\frac{dy}{y^{3/2}}\Pbb{\Oc_h}=o(1)\frac{dy}{y^{3/2}}\Pbb{\Oc_h},
\end{align}
where we have used \eqref{eq:localTau} and assumed that $y>4$ so that $y-u>y/2$. We note that this density in fact decays faster with factor $\frac{h}{\gN{h}}=o(1)$ than the required \eqref{eq:strongRepulsion}.

In the remaining two scenarios we employ that 
\[\lbcurlyrbcurly{\tau_T>2g\lbrb{h}}\cap\Oc_h \subset\lbcurlyrbcurly{\Delta_1\leq h,\,T=\Delta_1}\cap\Oc_h=\lbcurlyrbcurly{\Delta_1\leq h,\,\tau_{\Delta_1-}<g(h)}\cap\Oc_h,\]
which follows from the definitions of $T,\Delta_1,\Oc_h=\lbcurlyrbcurly{\tau_s>g(s),\,s\leq h}$ and the fact that $\tau$ is a subordinator, i.e. an increasing \LL process.

First, we consider $\sigma^{2}_h(dy):=\sigma_h\lbrb{dy,\Delta_1\leq h,\tau_{\Delta_1-}<g(h),S_\Delta<g(h)y/2}$ which majorizes in terms of measures $\sigma_h\lbrb{dy,\tau_T>2g\lbrb{h},S_\Delta<g(h)y/2}$. We disintegrate with respect to $\Delta_1$ and the position prior to the jump to get 
\begin{align*}
&\sigma^{2}_h(dy)=o^{2}_h(y)dy=\\
&\int_{s=0}^{h}\int_{w=0}^{1}\int_{v=1}^{y/2}\Pbb{\tau_{h-s}\in g(h)\lbrb{dy-v-w}}\Pbb{\Delta_1\in ds,\tau_{\Delta_1-}\in g(h)dw,S_\Delta\in g(h)dv;\,\Oc_h}.
\end{align*}
Using the definition of $\Oc_h=\lbcurlyrbcurly{\tau(s)>g(s),\,s\leq h}$, the fact that $g$ is an increasing function and $\tau$ is a subordinator we have the identity
\begin{align*}
&\Pbb{\Delta_1\in ds,\tau_{\Delta_{1}-}\in g(h)dw, S_\Delta\in g(h)dv;\,\Oc_h}=\\
&\quad\quad\quad\quad\quad \quad \Pbb{\Delta_1\in ds,\tau_{\Delta_{1}-}\in g(h)dw,S_\Delta\in g(h)dv;\,\Oc_{\Delta_{1}-}}.
\end{align*}
Since conditionally on $\lbcurlyrbcurly{\Delta_1=s}$ the jump $S_{\Delta}$ is independent of the past we get that
\begin{align*}
&\Pbb{\Delta_1\in ds,\tau_{s-}\in g(h)dw,S_\Delta\in g(h)dv;\,\Oc_{\Delta_{1}-}}=\\
&\quad\quad\quad\quad\quad \quad \Pbb{S_\Delta\in g(h)dv}\Pbb{\Delta_1\in ds,\tau_{\Delta_{1}-}\in g(h)dw;\,\Oc_{\Delta_{1}-}}=\\
&\quad\quad\quad\quad\quad \quad \Pbb{S_\Delta\in g(h)dv}\Pbb{\Delta_1\in ds,\tau_{\Delta_{1}-}\in g(h)dw;\,\Oc_h}
\end{align*}
Substituting this back above and using \eqref{eq:sizeOfJump} for the law of $S_\Delta$ we get that
\begin{align*}
&\sigma^{2}_h(dy)=o^{2}_h(y)dy=\\
&\int_{s=0}^{h}\int_{w=0}^{1}\int_{v=1}^{y/2}\frac{h-s}{\gN{h}\lbrb{y-w-v}^{3/2}}\Pbb{\Delta_1\in ds,\tau_{\Delta_{1}-}\in g(h)dw;\,\Oc_h}\frac{dv}{v^{3/2}}\,\,dy.
\end{align*}
 Using that $y-w-v\geq y/2-1\geq  y/3$ once $y>10$ we have that
\begin{equation}\label{eq:estimateRepuslion2}
\sigma^{2}_h(dy)=o^{2}_h(y)dy\leq \frac{h}{\sqrt{g(h)}}\frac{dy}{y^{3/2}}\Pbb{\Oc_h}=o(1)\frac{dy}{y^{3/2}}\Pbb{\Oc_h}.
\end{equation}
Secondly, we study the measure  $\sigma^{3}_h(dy):=\sigma_h\lbrb{dy,\Delta_1\leq h,\tau_{\Delta_1-}<g(h),S_\Delta>g(h)y/2}$, which majorizes in terms of measures $\sigma_h\lbrb{dy,\tau_T>2g\lbrb{h},S_\Delta>g(h)y/2}$.   We similarly disintegrate  the measure $\sigma^3_h(dy)$ to get
\begin{align*}
&\sigma^{3}_h(dy)=o^{3}_h(y)dy=\\
&\int_{s=0}^{h}\int_{w=0}^{1}\int_{v=y/2}^{\infty}\Pbb{\tau_{h-s}\in g(h)\lbrb{dy-v-w}}\Pbb{\Delta_1\in ds,\tau_{\Delta_{1}-}\in g(h)dw,S_\Delta\in g(h)dv;\,\Oc_h}\leq\\
&\int_{s=0}^{h}\int_{w=0}^{1}\int_{v=y/3-w}^{y-w}\Pbb{\tau_{h-s}\in g(h)\lbrb{dy-v-w}}\Pbb{\Delta_1\in ds,\tau_{\Delta_{1}-}\in g(h)dw;\,\Oc_h}\frac{dv}{v^{3/2}}.
\end{align*}
We note from \eqref{eq:localTau} that for $y>v+w$
\[\Pbb{\tau_{h-s}\in g(h)\lbrb{dy-v-w}}=\frac{1}{\sqrt{2\pi\frac{g(h)(y-v-w)}{(h-s)^2}}\lbrb{y-v-w}} e^{-\frac{1}{2\frac{g(h)(y-v-w)}{(h-s)^2}}}dy.\]
Put $a(h,s)=g(h)/(h-s)^2$. Change variables in $v$ such that $v\mapsto z-w $ and then $z\to y\rho$ in the last integral to get that
\begin{align*}
	&\sigma^{2}_h(dy)=o^{2}_h(y)dy\leq\\
	&\int_{s=0}^{h}\int_{w=0}^{1}\int_{\rho=1/3}^{1}\frac{1}{\sqrt{2\pi a(h,s)(y-y\rho)}\lbrb{y-y\rho} }e^{-\frac{1}{2a(h,s)(y-\rho)}}\Pbb{\Delta_1\in ds,\tau_{\Delta_{1}-}\in g(h)dw;\,\Oc_h}\frac{yd\rho}{\lbrb{y\rho-w}^{3/2}}dy
\end{align*}
Furthermore, using the inequalities  $y\rho-w>y\rho-1>y\rho/4$ once $y>20$ since $\rho\in(1/3,1)$ we obtain that 
\begin{align*}
&\sigma^{3}_h(dy)=o^{3}_h(y)dy\leq\\
&\int_{s=0}^{h}\int_{w=0}^{1}\int_{\rho=1/3}^{1}\frac{1}{\sqrt{2\pi a(h,s)y(1-\rho)}\lbrb{y(1-\rho)}} e^{-\frac{1}{2a(h,s)y(1-\rho)}} \times\\
&\frac{4^{3/2}d\rho}{\sqrt{y}\rho^{3/2}}\Pbb{\Delta_1\in ds,\tau_{\Delta_{1}-}\in g(h)dw;\,\Oc_h}\,dy\\
&\leq \frac{12^{3/2}}{y^{3/2}}\int_{s=0}^{h}\int_{w=0}^{1}\int_{\sigma=0}^{\frac{2}{3}}\frac{1}{\sqrt{2\pi a(h,s)y\sigma}\sigma} e^{-\frac{1}{2a(h,s)y\sigma}}d\sigma\Pbb{\Delta_1\in ds,\tau_{\Delta_{1}-}\in g(h)dw;\,\Oc_h}\,dy=\\
&\frac{12^{3/2}}{y^{3/2}}\int_{s=0}^{h}\int_{w=0}^{1}\int_{\chi=0}^{\frac{4}{3}ya(h,s)}\frac{1}{\sqrt{\pi \chi}\chi} e^{-\frac{1}{\chi}}d\chi\Pbb{\Delta_1\in ds,\tau_{\Delta_{1}-}\in g(h)dw;\,\Oc_h}\,dy.
\end{align*}
Since $\int_{\chi=0}^{\infty}\chi^{-3/2}e^{-1/\chi}d\chi<\infty$ we get that for $y>20$
\begin{equation}\label{eq:estimateRepuslion3}
\sigma^{3}_h(dy)=o^{3}_h(y)dy\leq C\frac{dy}{y^{3/2}}\Pbb{\Delta_1\leq h,\tau_{\Delta_1-}\leq g(h);\,\Oc_h}
\end{equation}
 By trivial estimates using \eqref{eq:estimateRepuslion1}, \eqref{eq:estimateRepuslion2} and  \eqref{eq:estimateRepuslion3} conclude  the upper bound for \eqref{eq:strongRepulsion}. For the lower bound consider that \[\sigma^3_h(dy)=\sigma_h\lbrb{dy,\Delta_1\leq h,\tau_{\Delta_1-}<g(h),S_\Delta>g(h)y/2}\leq \sigma_h(dy)\]
since $\lbcurlyrbcurly{\Delta_1\leq h,\tau_{\Delta_1-}<g(h),S_\Delta>g(h)y/2}\cap \Oc_h\subset\Oc_h$. 
Then the lower bound  for \eqref{eq:strongRepulsion} follows by observing that 
\begin{align*}
&\sigma^{3}_h(dy)=o^{3}_h(y)dy=\\
&\int_{s=0}^{h}\int_{w=0}^{1}\int_{v=y/2}^{\infty}\Pbb{\tau_{h-s}\in g(h)\lbrb{dy-v-w}}\Pbb{\Delta_1\in ds,\tau_{\Delta_{1}-}\in g(h)dw,S_\Delta\in g(h)dv;\,\Oc_h}\geq\\
&\int_{s=0}^{h}\int_{w=0}^{1}\int_{v=2y/3-w}^{y-w}\Pbb{\tau_{h-s}\in g(h)\lbrb{dy-v-w}}\Pbb{\Delta_1\in ds,\tau_{\Delta_{1}-}\in g(h)dw;\,\Oc_h}\frac{dv}{v^{3/2}}.
\end{align*}
once $2/3y-w\geq y/2$ which holds for $y>20$.
Then we feed in the expression for $\Pbb{\tau_{h-s}\in g(h)\lbrb{dy-v-w}}$, change in the same way the variables, estimate first $y\rho-w<y\rho$ and then instead of estimating from above $\rho^{-3/2}$ we estimate it from below with $1$ since $\rho\in\lbrb{2/3,1}$ to get similarly that
\begin{align*}
&\sigma^{3}_h(dy)=o^{3}_h(y)dy\geq\\
&\frac{1}{y^{3/2}}\int_{s=0}^{h}\int_{w=0}^{1}\int_{\chi=0}^{\frac{2}{3}ya(h,s)}\frac{1}{\sqrt{\pi \chi}\chi} e^{-\frac{1}{\chi}}d\chi\Pbb{\Delta_1\in ds,\tau_{\Delta_{1}-}\in g(h)dw;\,\Oc_h}\,dy.
\end{align*}
However, since $a(h,s)=\frac{g(h)}{\lbrb{h-s}^2}\geq \frac{g(h)}{h^2}\to\infty,$ as $h\to \infty$ we get with some $C>0$ that 
\begin{equation}\label{eq:estimateRepulsion4}
\sigma^{3}_h(dy)=o^{3}_h(y)dy\geq C\frac{dy}{y^{3/2}}\Pbb{\Delta_1\leq h,\tau_{\Delta_{1}-}\leq g(h);\,\Oc_h}.
\end{equation}
In the sense of measures
\[\sigma^{3}_h(dy)\leq\sigma_h\lbrb{dy}\leq \sigma^{1}_h(dy)+\sigma^{2}_h(dy)+\sigma^{3}_h(dy),\text{ $y>20$}.\]
If we assume that over a subsequence $h_i\uparrow\infty,$ $\Pbb{\Delta_1\leq h_i,\tau_{\Delta_{1}-}\leq g(h_i);\,\Oc_{h_i}}=o(1)\Pbb{\Oc_{h_i}},$ as $i\to\infty$ then \eqref{eq:estimateRepuslion1}, \eqref{eq:estimateRepuslion2} \eqref{eq:estimateRepuslion3} and \eqref{eq:estimateRepulsion4} yield that in sense of measures $\sigma_{h_i}\lbrb{dy}=o(1)\Pbb{\Oc_{h_i}}dy/y^{3/2}, y>20$, as $i\to\infty$. Therefore upon this assumption for $i$ big enough and $y>20$ we have that 
\[\Pbb{\tau_{h_i}>20g(h_i);\,\Oc_{h_i}}=\int_{20}^{\infty}\sigma_{h_i}(dy)=o(1)\Pbb{\Oc_{h_i}}\int_{B}^{\infty}\frac{dy}{y^{3/2}}=o(1)\Pbb{\Oc_{h_i}}.\]
Next we will provide a contradiction by showing that
\begin{equation*}
  \liminf_{h\to\infty}\frac{\Pbb{\tau_{h}>20g(h);\,\Oc_{h}}}{\Pbb{\Oc_h}}>0.
\end{equation*}
However, recalling the usual notation $\Delta^a_1=\inf\lbcurlyrbcurly{t>0:\,\tau_t-\tau_{t-}>a},\,a>0$ we see that since $\tau$ is a subordinator further it suffices to show that
\begin{equation}\label{eq:Last}
 \liminf_{h\to\infty}\frac{\Pbb{\tau_{h}>20g(h);\,\Oc_{h}}}{\Pbb{\Oc_h}}\geq \liminf_{h\to\infty}\frac{\Pbb{\Delta_1=\Delta^{20g(h)}_1\leq h;\,\Oc_{h}}}{\Pbb{\Oc_h}}>0.
\end{equation}
Then we get that
\begin{align*}
\Pbb{\Delta_1=\Delta^{20g(h)}_1\leq h;\,\Oc_{h}}=\int_{0}^{h}\Pbb{\Oc^{g(h)}_s}\Pbb{\Delta_1\in ds;\Delta_1=\Delta^{20g(h)}_1}.
\end{align*}
However, 
\begin{align*}
&\Pbb{\Delta_1\in ds;\Delta_1=\Delta^{20g(h)}_1}=\Pbb{\Delta_1\in ds;S_{\Delta}>20h}=\\
&\Pbb{\Delta_1\in ds}\Pbb{S_{\Delta}>20h}=\frac{\PP\lbrb{20g(h)}}{\PP\lbrb{g(h)}}\Pbb{\Delta_1\in ds}=C\Pbb{\Delta_1\in ds}.
\end{align*}
Feeding the last expression back above we get
\[\Pbb{\Delta_1=\Delta^{20g(h)}_1\leq h;\,\Oc_{h}}=C\Pbb{\Delta_1\leq h;\,\Oc_h}\sim C\Pbb{\Oc_h},\]
where the last follows from the first relation of \eqref{eq:asymptoticInFiniteCase-1} of Theorem \ref{thm:asymptoticInFiniteCase} and $C\in\lbrb{0,1}$ is an absolute constant. Therefore we conclude from \eqref{eq:Last} that
\[ \liminf_{h\to\infty}\frac{\Pbb{\tau_{h}>20g(h);\,\Oc_{h}}}{\Pbb{\Oc_h}}\geq C>0\]
and thus a contradiction is furnished.

 \end{proof}

\end{document}